\pdfoutput=1
\documentclass[11pt,letterpaper]{amsart}
\usepackage[english]{babel}
\usepackage{amsmath,amssymb}
\usepackage{amsthm,thmtools}
\usepackage{mathtools}
\mathtoolsset{centercolon}
\usepackage[shortcuts]{extdash}
\usepackage{tabu}
\usepackage{enumitem}
\usepackage[rgb]{xcolor}
\usepackage{tikz}
\usetikzlibrary{arrows,calc,cd}
\usepackage{float}
\usepackage{lmodern}
\usepackage{microtype}
\usepackage[colorlinks, linkcolor = blue, citecolor = blue, urlcolor = blue]{hyperref}
\usepackage[all]{hypcap} 

\newcommand{\myautoref}[2]{\hyperref[#2]{\autoref*{#1}\ref*{#2}}}
\newcommand{\myautorefstar}[2]{\hyperref[#1]{\autoref*{#1}#2}}
\newcommand{\relautoref}[1]{\hyperref[#1]{Relation~\ref*{#1}}}
\newcommand{\mysecref}[1]{\hyperref[#1]{\S\ref*{#1}}}

\newcommand{\hair}{\ifmmode\mskip1mu\else\kern0.08em\fi}  

\makeatletter
\newcommand{\leqnomode}{\tagsleft@true}
\newcommand{\reqnomode}{\tagsleft@false}
\makeatother

\hypersetup{hypertexnames=false}

\title{Asymmetric graphs with quantum symmetry}

\author[J.~van Dobben de Bruyn]{Josse van Dobben de Bruyn}
\address{Department of Applied Mathematics and Computer Science, Technical~University~of~Denmark, Richard Petersens Plads 324, 2800~Kongens~Lyngby, Denmark}
\curraddr{Charles University, Faculty of Mathematics and Physics, Department of Applied Mathematics, Malostranské náměstí 25, 118 00 Praha 1, Czech Republic}
\email{josse.van-dobben-de-bruyn@matfyz.cuni.cz}

\author[D.\hair{}E.~Roberson]{David E.~Roberson}
\address{Department of Applied Mathematics and Computer Science, Technical~University~of~Denmark, Richard Petersens Plads 324, 2800~Kongens~Lyngby, Denmark}
\address{QMATH, Department of Mathematical Sciences, University of Copenhagen, Universitetsparken 5, 2100 Copenhagen \O, Denmark}
\thanks{JvDdB and DER are supported by the Carlsberg Foundation Young Researcher Fellowship CF21-0682 -- ``Quantum Graph Theory".}
\email{dero@dtu.dk}

\author[S.~Schmidt]{Simon Schmidt}
\address{Faculty of Computer Science, Ruhr University Bochum, Universitätsstra{\ss}e 150, 44801 Bochum, Germany}
\thanks{SS was funded by the Deutsche Forschungsgemeinschaft (DFG, German Research Foundation) under Germany's Excellence Strategy -- EXC 2092 CASA -- 390781972.}
\email{s.schmidt@rub.de}

\makeatletter
\hypersetup{pdftitle=\@title}
\makeatother

\colorlet{Green}{green!70!black}
\colorlet{Red}{red!80!black}
\tikzset{vertex/.style={circle,fill,inner sep=1.32pt,outer sep=.3pt},emph/.style={inner sep=1.35pt}}

\declaretheorem[style=plain,numberwithin=section]{theorem}
\declaretheorem[style=plain,numberlike=theorem]{lemma}
\declaretheorem[style=plain,numberlike=theorem]{proposition}
\declaretheorem[style=plain,numberlike=theorem]{corollary}

\declaretheorem[style=definition,numberlike=theorem]{definition}
\declaretheorem[style=remark,numberlike=theorem]{remark}
\declaretheorem[style=definition,numberlike=theorem]{example}

\declaretheorem[style=plain,title={Theorem}]{AlphTheorem}

\newcommand{\N}{\mathbb{N}}
\newcommand{\C}{\mathbb{C}}
\newcommand{\Z}{\mathbb{Z}}
\newcommand{\bbS}{\mathbb{S}}

\DeclareMathOperator{\Aut}{Aut}
\DeclareMathOperator{\Qut}{Qut}
\newcommand{\symmdiff}{\mathbin{\triangle}}
\newcommand{\ab}{{\operatorname{ab}}}
\DeclareMathOperator{\PSL}{PSL}

\DeclareMathOperator\QIso{QIso}
\DeclareMathOperator{\decolor}{dc}
\newcommand{\refine}[2][]{\operatorname{ref}^{{#1}}(#2)}
\newcommand{\stabref}[1]{\operatorname{sref}(#1)}

\newcommand{\id}{\mathrm{id}}

\makeatletter
\newcommand{\letter}[1]{\@alph{#1}} 
\makeatother

\makeatletter
\newcommand\niton{\mathrel{\m@th\mathpalette\canc@l\owns}}
\newcommand\canc@l[2]{{\ooalign{$\hfil#1/\mkern1mu\hfil$\crcr$#1#2$}}}
\makeatother

\DeclareSymbolFont{bbold}{U}{bbold}{m}{n}
\DeclareSymbolFontAlphabet{\mathbbold}{bbold}
\newcommand{\one}{\ensuremath{\mathbbold{1}}}

\DeclareFontFamily{U}{mathb}{\hyphenchar\font45}
\DeclareFontShape{U}{mathb}{m}{n}{
<5> <6> <7> <8> <9> <10> gen * mathb
<10.95> mathb10 <12> <14.4> <17.28> <20.74> <24.88> mathb12
}{}
\DeclareSymbolFont{mathb}{U}{mathb}{m}{n}
\DeclareMathSymbol{\bigast}{2}{mathb}{"06}

\begin{document}

\begin{abstract}
	We present an infinite sequence of finite graphs with trivial automorphism group and non-trivial quantum automorphism group.
	These are the first known examples of graphs with this property. Moreover, to the best of our knowledge, these are the first examples of \emph{any} asymmetric classical space that has nontrivial quantum symmetries.
	
	Our construction is based on solution groups to (binary) linear systems, as defined by Cleve, Liu and Slofstra in the context of non-local games.
	We first show that the dual quantum group of every solution group occurs as the quantum automorphism group of some graph, and then construct an infinite sequence of systems whose solution groups are nontrivial perfect groups.
	This leads to the desired sequence of graphs.
	
	In addition to our main result, we prove a number of related results that allow us to answer several open problems from the literature.
	We prove a weak quantum analog of Frucht's theorem, namely that every finite classical group $\Gamma$ occurs as the quantum automorphism group of a finite graph.
	Combined with our main result, this shows that, for every finite group $\Gamma$, there are graphs $G_1$ and $G_2$ that both have classical automorphism group isomorphic to $\Gamma$ but one of them has quantum symmetry and the other does not.
	Therefore, the quantum automorphism group of a graph is never determined by its classical automorphism group, and there do not exist any ``quantum excluding groups".
\end{abstract}

\maketitle  

\section{Introduction}

In non-commutative geometry, investigation of the quantum symmetries of a (non-)commutative space has led to the study of \emph{quantum automorphism groups}, which are actually not groups but quantum groups.
Given that the term ``quantum group'' can mean one of several things, it is perhaps not surprising that even finding the right definition of the quantum automorphism group of a certain object can be a challenge.
In cases where the quantum automorphism group is defined, it is interesting to study its properties and compare it to the classical automorphism group.
A natural question is to which extent the two can differ, and a great deal of research on quantum automorphism groups is motivated by questions of this type.

Structures on which quantum automorphism groups have been defined (and shown to exist) include finite (non-)commutative Hausdorff spaces \cite{WanSn}, finite (quantum) groups \cite{BSS}, finite graphs \cite{QBan,QBic}, and infinite sets and graphs \cite{Voigt}.
Interestingly, the quantum automorphism group of a finite (quantum) group always coincides with its classical automorphism group \cite{KSW}, whereas this is generally not the case for sets and graphs.
This means that, given their richer structure, graphs in particular provide fertile ground for studying the peculiarities of quantum symmetry.

Questions regarding the difference between the classical and quantum automorphism group of a graph have been studied by several authors.
A graph is said to have \emph{quantum symmetry} if its classical and quantum automorphism group are different.
Which graphs have quantum symmetry has been studied, for instance, in \cite{BanBic,Schmidtthesis}.
Related to this, the following question came up in the doctoral thesis of the third author \cite[Section 8.1]{Schmidtthesis} (see also \cite[Problem 3.9]{Weber-survey}): Is there a graph with quantum symmetry and trivial automorphism group?
In this paper, we will answer this question in the positive by constructing a sequence of graphs having this property.

Our construction relies on techniques from quantum information theory.
In \cite{CM,CLS}, the authors studied perfect quantum strategies for linear constraint system (LCS) games in terms of the \emph{solution group} of the linear system.
Solution groups of LCS games have since been used to disprove the strong Tsirelson conjecture \cite{Slofstra-Tsirelson} and prove that the set of quantum correlations is not closed \cite{SlofstraQcor}, among other results. 
More recently, connections between solution groups and graph minor theory were studied in \cite{PRSS}.

The connection between quantum automorphism groups of graphs and solution groups of linear constraint systems was already recognized by the second and third author in \cite{RSsolution}.
There it was shown that for every homogeneous solution group, there is a vertex-- and edge-colored graph whose quantum automorphism group is the dual of the solution group.
We will refine the construction in this article, so that we only need vertex-colored graphs for this.
We then present a general decoloring procedure that constructs an uncolored graph $G'$ from an arbitrary vertex-colored graph $G$ such that $\Qut(G) \cong \Qut(G')$.
By applying some additional manipulations directly on the linear systems that do not change the corresponding solution group, we show that the dual of \emph{any} solution group (homogeneous or inhomogeneous) can be realized as the quantum automorphism group of an uncolored graph.

With the construction above in hand, we obtain a graph with quantum symmetry and trivial automorphism group whenever the homogeneous solution group is a non-trivial perfect group (i.e., its abelianization is the trivial group).
To find such a perfect homogeneous solution group, we construct a linear constraint system from all mutually commuting triples of elements of order $2$ that multiply to the identity in the alternating group $A_n$, for all $n \geq 7$.
This way, we obtain our main result:

\begin{AlphTheorem}[Main result]
	\label{thm:main1}
	There exists an infinite sequence of pairwise non-isomorphic graphs $\{G_i\}_{i=1}^\infty$ such that each $G_i$ has trivial automorphism group and non-trivial quantum automorphism group.
\end{AlphTheorem}

In other words, there exist asymmetric graphs with quantum symmetry.
This resolves a prominent open question in the study of quantum automorphism groups of graphs~\cite{Schmidtthesis,QFrucht,Weber-survey}.
Furthermore, to the best of our knowledge, these are the first known examples of \emph{any} kind of classical (i.e., commutative) spaces with no classical symmetry but which do have quantum symmetry.%
	\footnote{Here, we mean ``commutative space" in the broadest possible sense (for example, a topological space, a probability space, a group, a spin manifold, etc.), with the appropriate notion of quantum symmetry (depending on the context, e.g. quantum automorphisms/isometries/etc.).}

The quantum automorphism groups of the first few graphs in our sequence (the only ones that we were able to compute exactly) are finite\-/dimensional, but graphs with trivial automorphism group and infinite\-/dimensional quantum automorphism group can also be obtained from this by some straightforward manipulations.
See \autoref{rmk:infinite} for further details.

\subsection*{Additional results}
In addition to our main result (\autoref{thm:main1}), we also prove a number of related results which answer several other open questions about quantum automorphism groups of graphs.

First, in \cite[Problem 3.11]{Weber-survey}, it was asked whether or not there exist two connected, asymmetric graphs $G_1$ and $G_2$ that are quantum isomorphic but not classically isomorphic.
As pointed out in \cite{Weber-survey}, a positive solution to this problem would immediately give an asymmetric graph with quantum symmetry, namely the disjoint union $G_1 \sqcup G_2$.
By a modification of our proof of \autoref{thm:main1}, we also obtain a positive answer to this question.

\begin{AlphTheorem}
	\label{thm:main3}
	There exist two connected, asymmetric graphs $G_1$ and $G_2$ such that $G_1 \not\cong G_2$ but $G_1 \cong_q G_2$.
\end{AlphTheorem}

Second, in~\cite{QFrucht}, the authors raise the question of whether the graphs that arise in the proof of Frucht's theorem~\cite{Frucht} have quantum symmetry.
Though the choice of these graphs are non-canonical, we show that the specific graphs used in~\cite{Frucht} always have no quantum symmetry.
This proves what can be thought of as a weak quantum analog of Frucht's theorem: that every finite group can be realized as the \emph{quantum} automorphism group of a finite graph.
A stronger quantum analog of Frucht's theorem, that every quantum permutation group can be realized as the quantum automorphism group of some graph, was recently proven not to hold~\cite{QFrucht}. Nonetheless, by considering quantum graphs instead of graphs, it was shown in~\cite{brannan2025quantumfrucht} that every finite quantum group is the quantum automorphism group of a finite quantum graph.  

By combining this weak quantum analog of Frucht's theorem with \autoref{thm:main1}, we are able to prove the following:

\begin{AlphTheorem}
	\label{thm:main2}
	For every finite group $\Gamma$, there exist graphs $G_1$ and $G_2$ such that $\Aut(G_1) \cong \Aut(G_2) \cong \Gamma$, $\Qut(G_1) = \Aut(G_1)$ and $\Qut(G_2) \neq \Aut(G_2)$.
\end{AlphTheorem}

In other words, for every finite group $\Gamma$, the class of all graphs with classical automorphism group isomorphic to $\Gamma$ contains both graphs with and without quantum symmetry.
This allows us to answer some further open questions from the literature.
Specifically, both~\cite{Schmidtthesis} and~\cite{QFrucht} ask if there exist any \emph{quantum-excluding groups}, which are defined in the latter as groups $\Gamma$ such that $\Aut(G) \cong \Gamma$ implies $\Qut(G) \cong \Gamma$ (in other words, every graph with automorphism group $\Gamma$ has no quantum symmetry).
\autoref{thm:main2} shows that no such groups exist.
Moreover, \autoref{thm:main2} shows that for every finite group $\Gamma$, there are graphs $G_1$ and $G_2$ with $\Aut(G_1) \cong \Aut(G_2) \cong \Gamma$ but $\Qut(G_1) \not\cong \Qut(G_2)$.
This shows that the classical automorphism group of a graph never uniquely determines the quantum automorphism group.
This answers a question raised in~\cite{Schmidtthesis}.

Finally, we highlight one of the results we prove along the way, which we believe to be of independent interest.
When constructing graphs having certain prescribed quantum permutation groups as their quantum automorphism groups, it is often useful to begin by constructing a colored graph, as this makes it easier to control the quantum automorphisms of the constructed graph.
The following general ``decoloring procedure'' shows that it is always possible to construct equivalent uncolored graphs while preserving all properties related to (quantum) automorphisms and, more generally, (quantum) isomorphism:

\begin{AlphTheorem}
	\label{thm:main4}
	There is a polynomial time decoloring procedure
	\[ \decolor : \{\text{vertex-colored graphs}\} \to \{\text{connected uncolored graphs}\} \]
	that preserves \textup(quantum\textup) isomorphism and \textup(quantum\textup) automorphism groups;
	that is, for all vertex-colored graphs $G$ and $G'$, one has:
	\begin{enumerate}[label=(\alph*)]
		\item $G \cong   G'$ if and only if $\decolor(G) \cong   \decolor(G')$;
		\item $G \cong_q G'$ if and only if $\decolor(G) \cong_q \decolor(G')$;
		\item $\Aut(G) \cong \Aut(\decolor(G))$;
		\item $\Qut(G) \cong \Qut(\decolor(G))$.
	\end{enumerate}
\end{AlphTheorem}

This generalizes a similar (but different) decoloring procedure given by the second and third author in \cite{RSsolution}, which only applies to graphs with minimum degree $\geq 3$ and does not preserve the quantum automorphism group in general.%
	\footnote{The decoloring procedure from \cite{RSsolution} adds new constraints on the generators of the fundamental representation, so the quantum automorphism group is only preserved if these additional constraints were already satisfied in the original graph.}
Moreover, the decoloring procedure from \cite{RSsolution} was not shown to preserve quantum isomorphism in general.
On the other hand, the decoloring procedure from \cite{RSsolution} is more general in another respect, because it applies to (certain) vertex-- and edge-colored graphs, whereas \autoref{thm:main4} only applies to (all) vertex-colored graphs.

Apart from removing the colors, \autoref{thm:main4} also makes the graphs connected, so it could even be useful for graphs that are already uncolored.
For example, it follows that it is not necessary to require the graphs $G_1$ and $G_2$ in \autoref{thm:main3} to be connected, because any disconnected example can be made connected by applying the decoloring procedure.

\subsection*{Organization of the paper}

\mysecref{sec:preliminaries} covers the preliminaries.

In \mysecref{sec:graph-construction}, we construct, for every linear system $Mx = b$ over $\Z_2$, a vertex-colored graph $G(M,b)$ such that $\Qut(G(M,b))$ is isomorphic (as a compact quantum group) to the dual of the homogeneous solution group $\Gamma(M,0)$.
This is a simplification of an earlier construction given by the second and third author in a previous paper \cite{RSsolution}, which used both vertex and edge colors and as a result could only be decolored under certain conditions.

In \mysecref{sec:decolor}, we provide a general decoloring procedure that turns vertex-colored graphs into uncolored graphs while preserving quantum isomorphism and the quantum automorphism group (see \autoref{thm:main4}).

In \mysecref{sec:uncolored}, we combine the results from \mysecref{sec:graph-construction} and \mysecref{sec:decolor} to construct, for every linear system $Mx = b$ over $\Z_2$, an uncolored graph $G'(M,b)$ such that $\Qut(G'(M,b))$ is isomorphic (as a compact quantum group) to the dual of the homogeneous solution group $\Gamma(M,0)$.
Furthermore, by applying certain modifications to the linear system, we show that the dual of every \emph{inhomogeneous} solution group can also be obtained as the quantum automorphism group of a graph.

In \mysecref{sec:perfect}, we construct a sequence of asymmetric graphs with quantum symmetry, thereby proving our main result (\autoref{thm:main1}).
To do so, we construct a sequence of linear systems whose homogeneous solution groups are perfect (i.e., have trivial abelianizations).
This is our main contribution on the side of solution groups, since no such linear systems were known before.

In \mysecref{sec:asym-qisom}, we construct a pair of asymmetric graphs that are quantum isomorphic but not classically isomorphic, thereby proving \autoref{thm:main3}.
We proceed in a similar fashion as in \mysecref{sec:perfect}, except that we need a solution group with the (stronger) property that the \emph{inhomogeneous} solution group is perfect.
Also, a little extra work is needed to show that the construction from \mysecref{sec:graph-construction} behaves well with respect to quantum isomorphism.

Finally, in \mysecref{sec:frucht}, we prove that the graphs in Frucht's original proof of Frucht's theorem \cite{Frucht} never have quantum symmetry, and we combine this with our main result to prove \autoref{thm:main2} and its consequences.

\section{Preliminaries}
\label{sec:preliminaries}

We denote by $\Z_2$ the integers modulo $2$.
We sometimes view this as an abelian group and sometimes as a finite field.

We now define compact quantum groups, which generalize compact groups.
We will work with the $C^*$\nobreakdash-algebraic definition by Woronowicz \cite{CQG}; see \cite{Timmermann,NeshTu} for more information about this topic.

\begin{definition}
	\label{def:CQG}
	A \emph{compact quantum group} $\mathbb{G}$ is a pair $(C(\mathbb{G}), \Delta)$, where $C(\mathbb{G})$ is a unital $C^*$-algebra and $\Delta: C(\mathbb{G}) \to C(\mathbb{G}) \otimes C(\mathbb{G})$ is a unital $*$-homomorphism with the properties
	\begin{itemize}
		\item $(\Delta \otimes \id) \circ \Delta = (\id \otimes \Delta) \circ \Delta$ (coassociativity);
		\item $\Delta(C(\mathbb{G}))(1 \otimes C(\mathbb{G}))$ and $\Delta(C(\mathbb{G}))(C(\mathbb{G}) \otimes 1)$ are dense in $C(\mathbb{G}) \otimes C(\mathbb{G})$ (cancellation property).
	\end{itemize}
	Here $\mathcal{A}\otimes \mathcal{B}$ denotes the minimal (i.e.{} spatial) tensor product of the $C^*$\nobreakdash-algebras $\mathcal{A}$ and $\mathcal{B}$.
\end{definition}

Every compact (classical) group $\Gamma$ is a compact quantum group by setting $C(\mathbb{G}) = C(\Gamma)$ and $\Delta(f)(g,h) = f(gh)$ for all $f \in C(\Gamma)$ and all $g,h \in \Gamma$.
Conversely, every compact quantum group with $C(\mathbb{G})$ commutative is of this form (e.g.~\cite[Example 1.1.2]{NeshTu}).
A more advanced example of a compact quantum group, which plays a crucial role in this paper, is the following:

\begin{example}[\cite{CMQG1}] 
	\label{Dualqgroup}
	Consider a discrete group $\Gamma$. The associated group $C^*$-algebra $C^*(\Gamma)$ can be written as 
	\begin{align*}
		C^*(\Gamma)=C^*(u_g, g\in \Gamma\,|\, u_g \text{ unitary, }u_{gh}=u_gu_h, u_{g^{-1}}=u^*_g).
	\end{align*}
	Let $\Delta$ be the $*$-homomorphism $\Delta: C^*(\Gamma) \to C^*(\Gamma) \otimes C^*(\Gamma)$, $\Delta(u_g) = u_g \otimes u_g$.
	Then the pair $(C^*(\Gamma), \Delta)$ is a compact quantum group, known as the \emph{\textup(quantum group\textup) dual} of $\Gamma$, denoted $\hat{\Gamma}$.
	If $\Gamma$ is abelian, then $C^*(\Gamma)$ is commutative and $\hat{\Gamma}$ coincides with the Pontryagin dual of $\Gamma$.
\end{example}

Throughout this paper, we will work exclusively with compact quantum groups that are \emph{universal} in the sense of \cite[Definition 5.4.9]{Timmermann} (i.e., $C(\mathbb{G})$ is the universal $C^*$\nobreakdash-algebra of the dense Hopf $*$\nobreakdash-algebra $C[\mathbb{G}]$).
For universal compact quantum groups, isomorphism can be defined as follows.

\begin{definition}
	Let $\mathbb{G}=(C(\mathbb{G}), \Delta_{\mathbb{G}})$ and $\mathbb{H}=(C(\mathbb{H}), \Delta_{\mathbb{H}})$ be universal compact quantum groups. We say that $\mathbb{G}$ and $\mathbb{H}$ are \emph{isomorphic as compact quantum groups} if there exists a $*$-isomorphism $\varphi:C(\mathbb{G})\to C(\mathbb{H})$ such that $\Delta_{\mathbb{H}}\circ \varphi=(\varphi \otimes \varphi)\circ \Delta_{\mathbb{G}}$. In this case, we write $\mathbb{G}\cong \mathbb{H}$.
\end{definition}

A special case of compact quantum groups are so called compact matrix quantum groups. Those generalize all compact matrix groups $\mathrm{G}\subseteq GL_n(\C)$.

\begin{definition}[\cite{CMQG1}]
	A \emph{compact matrix quantum group} $\mathbb{G}$ is a pair $(C(\mathbb{G}),u)$, where $C(\mathbb{G})$ is a unital $C^*$-algebra and $u = (u_{ij}) \in M_n(C(\mathbb{G}))$ is a matrix such that 
	\begin{itemize}
		\item the elements $u_{ij}$, $1 \le i,j \le n$, generate $C(\mathbb{G})$, 
		\item the $*$-homomorphism $\Delta: C(\mathbb{G}) \to C(\mathbb{G}) \otimes C(\mathbb{G})$, $u_{ij} \mapsto \sum_{k=1}^n u_{ik} \otimes u_{kj}$ exists,
		\item the matrix $u$ and its transpose $u^{T}$ are invertible. 
	\end{itemize}
	The matrix $u$ is usually called the \emph{fundamental representation} of $\mathbb{G}$.
\end{definition}

The following compact matrix quantum group was first discovered by Wang and can be considered as a quantum analogue of the symmetric group $S_n$.

\begin{definition}[\cite{WanSn}]
	The \emph{quantum symmetric group} $\bbS_n^+= (C(\bbS_n^+),u)$ is the compact matrix quantum group, where
	\begin{align*}
		C(\bbS_n^+) := C^*(u_{ij}, \, 1 \le i,j \le n \, | \, u_{ij} = u_{ij}^* = u_{ij}^2, \, \sum_{k=1}^n u_{ik} = \sum_{k=1}^n u_{kj} = 1). 
	\end{align*} 
\end{definition}

A matrix $u = [u_{ij}]_{i,j \in [n]}$ with entries from a nonzero unital $C^*$-algebra satisfying $u_{ij} = u_{ij}^* = u_{ij}^2$ and $\sum_{k=1}^n u_{ik} = \sum_{k=1}^n u_{kj} = 1$, as in the definition above, is known as a \emph{magic unitary}.

We are also dealing with free products of quantum groups, thus we need the following definition.\leavevmode
\begin{definition}[{\cite{wang1995free}}]
\label{def:freeprod}
Let $\mathbb{G}=(C(\mathbb{G}), \Delta_{\mathbb{G}})$ and $\mathbb{H}=(C(\mathbb{H}), \Delta_{\mathbb{H}})$ be compact quantum groups. Then, the \emph{free product} $\mathbb{G}\ast \mathbb{H}$ is the compact quantum group consisting of the $C^*$-algebra $C(\mathbb{G}\ast \mathbb{H})=C(\mathbb{G})\ast C(\mathbb{H})$ and the comultiplication 
\begin{align*}
	\Delta_{\mathbb{G}\ast \mathbb{H}}(a)=\Delta_{\mathbb{G}}(a)\quad \text{and}\quad \Delta_{\mathbb{G}\ast \mathbb{H}}(b)=\Delta_{\mathbb{H}}(b)\quad\text{for}\quad a\in C(\mathbb{G}), b\in C(\mathbb{H}).
\end{align*}
Here $C(\mathbb{G})\ast C(\mathbb{H})$ denotes the unital free product of $C^*$-algebras\footnote{Some authors also refer to this as the \emph{free product algamanted over the scalars}.} (see e.g.{} \cite[\S II.8.3.4]{Blackadar-encyclopedia}).
In the case of two compact matrix quantum groups $\mathbb{G}=(C(\mathbb{G}), u)$ and $\mathbb{H}=(C(\mathbb{H}), v)$, we have $\mathbb{G}\ast \mathbb{H}=(C(\mathbb{G})\ast C(\mathbb{H}),u \oplus v)$ (see \cite[Corollary 3.6]{wang1995free}).

\end{definition}

Throughout this paper, we use the following conventions for graphs.
A \emph{graph} is a tuple $G = (V(G),E(G))$ consisting of a finite vertex set $V(G)$ and an edge set $E(G) \subseteq V(G) \times V(G)$ such that $(i,i) \notin E(G)$ for all $i \in V(G)$ and $(i,j) \in E(G)$ if and only if $(j,i) \in E(G)$.
In other words, all graphs are undirected and do not have loops nor parallel edges.
We write $i \sim j$ or $(i,j)\in E(G)$ if the vertices $i$ and $j$ are adjacent and $i \nsim j$ or $(i,j)\notin E(G)$ if they are not.

A \emph{vertex-- and edge-colored graph} is a graph $G$ along with a \emph{coloring function} $c: V(G) \cup E(G) \to S$ for some set $S$.
We refer to $c(x)$ as the color of the vertex or edge $x$ and use $E_c(G)$ to refer to the set of edges of color $c$.
A \emph{vertex-colored graph} is a graph $G$ along with a coloring function $c : V(G) \to S$ for some set $S$.
Every graph can be viewed as a vertex-colored graph where all vertices have the same color, and every vertex-colored graph can be viewed as a vertex-- and edge-colored graph where all edges have the same color.

\begin{definition}[\cite{RSsolution}]
	\label{def:QAG}
	Let $G$ be a vertex-- and edge-colored graph. The \emph{quantum automorphism group} $\Qut(G)$ is the compact matrix quantum group $(C(\Qut(G)), u)$, where $C(\Qut(G))$ is the universal $C^*$-algebra with generators $u_{ij}$, $i,j \in V(G)$ and relations
	\begin{align}
		&u_{ij}=u_{ij}^*=u_{ij}^2, &&i,j \in V(G),\label{rel1}\\
		&\sum_{l} u_{il}= 1 = \sum_{l} u_{li}, &&i\in V(G),\label{rel2}\\
		&u_{ij}=0, &&\text{for all  } i,j \in V(G) \text{ with } c(i)\neq c(j),\label{rel3}\\
		&u_{ij}u_{kl}=0, &&\text{if $(i,k)\in E_c(G)$ and $(j,l)\notin E_c(G)$}\label{rel4}\\
		& &&\text{or vice versa for some edge-color $c$.}\nonumber
	\end{align}
\end{definition}

The automorphism group $\Aut(G)$ of a vertex-- and edge-colored graph $G$ is defined similarly as the group of color-preserving automorphisms of $G$.
The following lemma, whose proof we leave to the reader, shows that $\Aut(G)$ can be recovered from $\Qut(G)$:

\begin{lemma}
	\label{lem:ab}
	Let $G$ be a vertex-- and edge-colored graph.
	Then $C(\Aut(G))$ is isomorphic to the abelianization of $C(\Qut(G))$.
	In particular, if $\Qut(G) \cong \hat{\Gamma}$ for some discrete group $\Gamma$, then $\Aut(G) \cong \widehat{\Gamma_\ab}$, where $\Gamma_\ab$ is the abelianization of $\Gamma$.
\end{lemma}

Quantum isomorphisms of graphs \cite{nonsignalling} generalize graph isomorphisms in a similar way as quantum automorphism groups generalize automorphism groups of graphs. We need the vertex-colored version in this paper: 

\begin{definition}
	\label{def:qiso}
	Let $G$ and $G'$ be vertex-colored graphs whose colorings $c : V(G) \to C$, $c' : V(G') \to C$ take values in a common set of colors $C$, and let $A_G$ and $A_{G'}$ denote their adjacency matrices.
	We say that $G$ and $G'$ are \emph{quantum isomorphic}, denoted $G \cong_q G'$,\footnote{Caveat: this notation is not consistent with common usage in quantum information theory. In particular, in \cite{Brannan-bigalois}, our notion of quantum isomorphism is written $G \cong_{C^*} G'$, and there the notation $G \cong_q G'$ is reserved for the case that there exists a \emph{finite\-/dimensional} $C^*$-algebra $\mathcal A$ as in \autoref{def:qiso}. However, in this paper, we only consider one type of quantum isomorphism, which we denote $G \cong_q G'$ (instead of $G \cong_{C^*} G'$) for simplicity.} if there is a nonzero unital $C^*$\nobreakdash-algebra $\mathcal{A}$ and a magic unitary $u = [u_{ij}]_{i \in V(G), j \in V(G')}$ with entries from $\mathcal{A}$ such that $A_Gu = uA_{G'}$ and $u_{ij} = 0$ whenever $c(i) \neq c'(j)$.
\end{definition}

Note that taking a vertex-colored graph $G$ and permuting the colors results in a graph that is generally \underline{not} quantum isomorphic to $G$.
In fact, for two graphs to be quantum isomorphic, it is necessary that they have the same number of vertices of color $r$, for every color $r \in C$ (see \autoref{rmk:qiso-vertex-counts}).

\section{A vertex-colored graph whose quantum automorphism group is the dual of a solution group}
\label{sec:graph-construction}

Below we give the definition of a solution group associated to a linear system $Mx=b$ over $\mathbb{Z}_2$ from~\cite{CLS}. We will use $1$ to denote the identity element of a group.

\begin{definition}
	\label{def:solgroup}
	Let $M \in \mathbb{Z}_2^{m \times n}$ and $b \in \mathbb{Z}_2^m$ with $b \ne 0$. The \emph{solution group} $\Gamma(M,b)$ of the linear system $Mx = b$ is the finitely presented group generated by elements $x_i$ for $i \in [n]$ and an element $\gamma$ satisfying the following relations:
	\begin{enumerate}
		\item $x_i^2 = 1$ for all $i \in [n]$\label{solgroup:rel1};
		\item $x_ix_j = x_jx_i$ if there exists $k \in [m]$ s.t. $M_{ki} = M_{kj} = 1$;\label{solgroup:rel2}
		\item $\prod_{i: M_{ki} = 1} x_i = \gamma^{b_k}$ for all $k \in [m]$;
		\item $\gamma^2 = 1$;
		\item $x_i\gamma = \gamma x_i$ for all $i \in [n]$.
	\end{enumerate}
	
	If $b = 0$, then we refer to $\Gamma(M,b)$ as the \emph{homogeneous solution group} of the system $Mx = 0$, and define this the same as above except that we add the relation $\gamma = 1$. This is equivalent to removing $\gamma$ from the list of generators, changing the righthand side of the equation in (3) to $1$, and removing items (4) and (5).

	\begin{remark}\label{rem:unigrpalg}
		Let $\Gamma=\Gamma(M,0)$ be the homogeneous solution group of the system $Mx=0$. By \autoref{Dualqgroup}, the group $C^*$-algebra $C^*(\Gamma)$ can be written as 
		\begin{align*}
			C^*(\Gamma)=C^*(x_i, i \in [n]\,|\, &x_i=x_i^*, x_i^2 = 1, x_ix_j = x_jx_i \text{ if there exists $k \in [m]$}\\&\text{ s.t. } M_{ki} = M_{kj} = 1,\prod_{i: M_{ki} = 1} x_i = 1\text{ for all $k \in [m]$}).
		\end{align*}
		Note that we slightly abuse the notation by writing $x_i$ instead of $u_{x_i}$.
	\end{remark}
	
	\begin{remark}
		Let $M\in \mathbb{Z}_2^{m \times n}$ be a linear system with some empty columns, i.e., variables that do not appear in any equation. Reordering the columns, we have
		\begin{align*}
			M=\begin{pmatrix} M' & 0_{m,s}\end{pmatrix},
		\end{align*}
		where $M'\in \mathbb{Z}_2^{m \times n-s}$ and $0_{m,s}\in \mathbb{Z}_2^{m \times s}$ for some $s\leq n$. We see from the previous remark that
		\begin{align*}
			C^*(\Gamma(M,0))&=C^*(\Gamma(M',0))\ast C^*(x_i, i\in [s]\,|\, x_i=x_i^*, x_i^2 = 1)\\
			&=C^*(\Gamma(M',0))\ast C^*(\widehat{\Z_2^{\ast s}}).
		\end{align*}
		Looking at \autoref{def:freeprod}, we see that $\hat{\Gamma}(M,0)=\hat{\Gamma}(M',0)\ast \widehat{\Z_2^{\ast s}}$.
	\end{remark}

	We will denote by $S_k(M)$ the set $\{i \in [n] : M_{ki} = 1\}$ and by $T_i(M)$ the set $\{k \in [m] : M_{ki} = 1\}$, often writing simply $S_k$ and $T_i$ when $M$ is clear from context. We use $\pm 1^S$ to denote the set of functions $\alpha : S \to \{1,-1\}$, and will typically write $\alpha_i$ instead of $\alpha(i)$. We will also use $\pm 1^S_0$ to denote the subset of such functions satisfying $\prod_{i \in S} \alpha_i = 1$, and similarly use $\pm 1^S_1$ for the set of such functions satisfying $\prod_{i \in S} \alpha_i = -1$. Lastly, given $k,l \in [m]$ such that $S_k \cap S_l \ne \varnothing$, for any $\alpha \in \pm 1^{S_k}$ and $\beta \in \pm 1^{S_l}$ we define the function $\alpha \symmdiff \beta \in \pm 1^{S_k \cap S_l}$ as $(\alpha \symmdiff \beta)_i = \alpha_i \beta_i$. Note that
	\[ (\alpha \symmdiff \beta)_i = \alpha_i\beta_i = \begin{cases}
		1 & \text{if } \alpha_i = \beta_i \\
		-1 & \text{if } \alpha_i \ne \beta_i
	\end{cases} \]
\end{definition}

In~\cite{RSsolution}, the authors introduced a vertex-- and edge-colored graph $\widehat{G}(M,b)$ whose quantum automorphism group is isomorphic to the dual of $\Gamma(M,0)$. In order to get an uncolored graph whose quantum automorphism group is isomorphic to the dual of $\Gamma(M,0)$, the authors of~\cite{RSsolution} use a vertex and edge decoloring procedure that produces an uncolored graph with quantum automorphism group isomorphic to that of the colored graph. However, the edge decoloring procedure was quite limited in that it only worked on graphs $G$ where certain entries of the fundamental representation of $\Qut(G)$ satisfy an additional nontrivial commutativity requirement. In this work we remedy this limitation by defining a similar graph $G(M,b)$ which is only vertex-- but not edge-colored but whose quantum automorphism group is isomorphic to the dual of $\Gamma(M,0)$. Since the vertex decoloring procedure of~\cite{RSsolution} works on every vertex-colored graph with minimum degree at least three, we can apply this to $G(M,b)$ for all linear systems $Mx=b$ with $|S_k|\geq 3$ for all $k$.

We now give the definition of $\widehat{G}(M,b)$ from~\cite{RSsolution}\footnote{In~\cite{RSsolution} the graph was denoted $G(M,b)$.}:

\begin{definition}
	\label{def:graph}
	Let $M \in \mathbb{Z}_2^{m \times n}$ and $b \in \mathbb{Z}_2^m$ such that $|T_i|\geq 1$ for all $i\in [n]$. Define the colored graph $\widehat{G} := \widehat{G}(M,b)$ as follows. The vertex set of $\widehat{G}$ is $\left\{(k,\alpha): k \in [m], \ \alpha \in \pm 1^{S_k}_{b_k}\right\}$. The color of a vertex $v = (k,\alpha)$, denoted $c(v)$, is $k$. We also define the subsets $Q_k = \{v \in V(\widehat{G}) : c(v) = k\}$ for all $k \in [m]$, and these partition the vertex set. The vertices of $Q_k$ correspond to the \emph{satisfying} assignments of values to the variables appearing (with nonzero coefficient) in equation $k$ of $Mx=b$ (except that values of 1 and $-1$ are used instead of 0 and 1). For any $l,k \in [m]$ such that $S_l \cap S_k \ne \varnothing$, the graph $\widehat{G}$ contains all (non-loop) edges between $Q_l$ and $Q_k$ (thus each $Q_k$ induces a complete subgraph). Given an edge $e$ between adjacent vertices $(l,\alpha)$ and $(k,\beta)$, the color of $e$, denoted $c(e)$, is equal to the function $\alpha \symmdiff \beta \in \pm 1^{S_l \cap S_k}$.
\end{definition}

\begin{remark}
	We need the assumption $|T_i|\geq 1$ for all $i\in [n]$ (i.e., that $M$ has no zero columns), since our graph construction does not keep track of any variable that does not appear in any equation. This assumption should have already been mentioned explicitly in \cite{RSsolution}. We will give a different graph construction later in this section, which also keeps track of variables that do not appear in any equation. 
\end{remark}

The following theorem was proven in~\cite{RSsolution} and states that $\Qut(\widehat{G}(M,b))$ is the dual of the solution group $\Gamma(M,0)$.

\begin{theorem}[\cite{RSsolution}]\label{thm:qaut2cgamma}
	Let $M \in \mathbb{Z}_2^{m \times n}$ and $b \in \mathbb{Z}_2^m$ such that $|T_i|\geq 1$ for all $i\in [n]$. Set $\widehat{G} = \widehat{G}(M,b)$ and $\Gamma = \Gamma(M,0)$. Let $\Delta_\Gamma$ and $\Delta_{\widehat{G}}$ denote the comultiplications on $C^*(\Gamma)$ and $C(\Qut(\widehat{G}))$ respectively. Then there exists an isomorphism $\varphi : C^*(\Gamma) \to C(\Qut(\widehat{G}))$ such that $\Delta_G \circ \varphi = (\varphi \otimes \varphi) \circ \Delta_\Gamma$. In other words, $\Qut(\widehat{G})$ is isomorphic to the dual of the solution group $\Gamma$.
\end{theorem}

It is easy to see that due to the vertex colors, the fundamental representation of $\Qut(\widehat{G}(M,b))$ has a block diagonal form as expressed in the following lemma from~\cite{RSsolution}:

\begin{lemma}
	Let $M \in \mathbb{Z}_2^{m \times n}$ and $b \in \mathbb{Z}_2^m$ such that $|T_i|\geq 1$ for all $i\in [n]$. Set $\widehat{G} = \widehat{G}(M,b)$. If $\hat{u}$ is the fundamental representation of $\Qut(G)$, then $\hat{u}$ has the form
	\[\bigoplus_{k \in [m]} \hat{q}^{(k)}\]
	where each $\hat{q}^{(k)}$ is a magic unitary indexed by $Q_k$.
\end{lemma}

Note that we will consistently use $\hat{u}$ to denote the fundamental representation of $\Qut(\widehat{G}(M,b))$, and we will use and $\hat{q}^{(k)}$ to denote the sub magic unitaries from the lemma above.

We now define our vertex-colored graph $G(M,b)$ which is similar to $\widehat{G}(M,b)$, but has additional vertices and a different edge set.

\begin{definition}
	\label{defgraph}
	Let $M \in \mathbb{Z}_2^{m \times n}$ and $b \in \mathbb{Z}_2^m$. Define the vertex-colored graph $G := G(M,b)$ as follows. The vertex set of $G$ is
	\[\big\{(i,s) : i \in [n], \ s \in \{\pm 1\} \big\} \ \bigcup \ \left\{(k,\alpha): k \in [m], \ \alpha \in \pm 1^{S_k}_{b_k}\right\}.\]
	As in the definition of $\widehat{G}(M,b)$, the vertices on the right above represent the satisfying assignments of values to the variables appearing with nonzero coefficient in a given equation of $Mx=b$, whereas the vertices in the set on the left represent the possible assignments of values to a given variable $x_i$ in the system $Mx=b$. The color of a vertex $(i,s)$ of the latter type denoted is $(i,0)$ and the color of a vertex $(k, \alpha)$ of the former type is $(k,1)$. We denote the color of a vertex $v$ by $c(v)$. We also define the subsets $Q_k = \{v \in V(G) : c(v) = (k,1)\}$ for all $k \in [m]$ and $V_i = \{v \in V(G) : c(v) = (i,0)\}$ for all $i \in [n]$, and these partition the vertex set. The notation $V_i$ is meant to refer to variable $x_i$, and $Q_k$ is meant to refer to the $k^\text{th}$ equation in the system $Mx=b$. For any $k \in [m]$ and $i \in S_k$, the graph $G$ has an edge between $(k,\alpha)$ and $(i,\alpha_i)$. In other words, there is an edge between a vertex corresponding to an assignment to variable $x_i$ and a vertex corresponding to an assignment for equation $k$ if they agree on the value assigned to $x_i$.
\end{definition}

\begin{remark}
	Note that the sets $Q_k$ defined above are precisely the sets $Q_k$ from the definition of $\widehat{G}(M,b)$. Thus the vertices of $\widehat{G}(M,b)$ are simply the vertices of $G(M,b)$ that correspond to equations.
\end{remark}

Before working with $G(M,b)$, we state a lemma and a corollary we need later on.

\begin{lemma}
	\label{lem:resofid}
	Let $\mathcal{A}$ be a unital $C^*$-algebra and suppose that $a_1, \ldots, a_n \in \mathcal{A}$ and $b_1, \ldots, b_n \in \mathcal{A}$ are all projections such that $\sum_{i=1}^n a_i = 1$ and $\sum_{i=1}^n b_i = 1$. Then the following are equivalent
	\begin{enumerate}
		\item $a_ib_i = a_i$ for all $i \in [n]$;
		\item $a_ib_i = b_i$ for all $i \in [n]$;
		\item $a_ib_j = 0$ if $i \ne j$;
		\item $a_i = b_i$ for all $i \in [n]$.
	\end{enumerate}
\end{lemma}
\noindent
The proof is routine and follows from basic properties of projections in $C^*$-algebras.

\begin{corollary}
	\label{cor:resofid}
	Let $\mathcal{A}$ be a unital $C^*$-algebra and suppose that $a_1, \ldots, a_n \in \mathcal{A}$ are projections such that $\sum_{i=1}^n a_i = 1$ and $b_1, \ldots, b_n \in \mathcal{A}$ are pairwise orthogonal projections such that $a_ib_i = a_i$ for all $i \in [n]$. Then $\sum_{i=1}^n b_i = 1$ and therefore $a_i = b_i$ for all $i \in [n].$
\end{corollary}
\begin{proof}
	First, for $i \ne j$ we have that $a_ib_j = a_ib_ib_j = 0$. Therefore,
	\[\sum_{j=1}^n b_j = \left(\sum_{i=1}^n a_i\right) \left(\sum_{j=1}^n b_j\right) = \sum_{i=1}^n a_ib_i = \sum_{i=1}^n a_i = 1.\]
	Thus it follows from \autoref{lem:resofid} that $a_i = b_i$ for all $i \in [n]$.
\end{proof}

We will now have a look at the quantum automorphism group of $G(M,b)$. As with $\widehat{G}(M,b)$, it is immediate from the vertex colors of $G(M,b)$ that the fundamental representation of its quantum automorphism group has a block diagonal form described as follows:

\begin{lemma}
	\label{lem:qutGdecomp}
	Let $M \in \mathbb{Z}_2^{m \times n}$ and $b \in \mathbb{Z}_2^m$. Set $G = G(M,b)$. If $u$ is the fundamental representation of $\Qut(G)$, then $u$ has the form
	\[\left(\bigoplus_{i \in [n]} v^{(i)}\right) \oplus \left(\bigoplus_{k \in [m]} q^{(k)}\right)\]
	where $v^{(i)}$ and $q^{(k)}$ are magic unitaries indexed by $V_i$ and $Q_k$ respectively.
\end{lemma}

Note that the magic unitaries $v^{(i)}$ for $i \in [m]$ are $2 \times 2$ magic unitaries since $|V_i| = 2$. Therefore, we have that
\begin{equation}\label{eq:2x2}
	v^{(i)}_{(i,1),(i,1)} = v^{(i)}_{(i,-1),(i,-1)} = 1 - v^{(i)}_{(i,1),(i,-1)} = 1 - v^{(i)}_{(i,-1),(i,1)}.
\end{equation}
To shorten our notation we will thus use $v^{(i)}_{1}$ to refer to the above and we will set
\[v^{(i)}_{-1} := 1-v^{(i)}_1 = v^{(i)}_{(i,1),(i,-1)} = v^{(i)}_{(i,-1),(i,1)}.\]
Note that this means that
\begin{equation}
	v^{(i)}_s = v^{(i)}_{(i,t),(i,st)} \text{ for all } s,t \in \{\pm 1\}. \label{eq:vis}
\end{equation}
We will also use $q^{(k)}_{\alpha,\beta}$ to denote $u_{(k,\alpha),(k,\beta)}$. 

We now show several relations among the generators of $\Qut(G(M,b))$.

\begin{lemma}\label{lem:qvprod}
	Let $M \in \mathbb{Z}_2^{m \times n}$ and $b \in \mathbb{Z}_2^m$. Set $G = G(M,b)$ and let $u$ be the fundamental representation of $\Qut(G)$. For any $k \in [m]$ and $i \in S_k$, we have that
	\begin{equation}
		q^{(k)}_{\alpha,\beta}v^{(i)}_s = v^{(i)}_sq^{(k)}_{\alpha,\beta} = \begin{cases}
			q^{(k)}_{\alpha,\beta} & \text{if } \alpha_i\beta_i = s \\
			0 & \text{otherwise}
		\end{cases}
	\end{equation}
\end{lemma}
\begin{proof}
	Suppose that $\alpha_i\beta_i \ne s$. Then $\beta_i \ne s\alpha_i$. We have that $(k,\alpha) \sim (i,\alpha_i)$, $(k,\beta) \sim (i,\beta_i)$, but $(k,\beta) \nsim (i,s\alpha_i)$. Thus, applying Equation~\eqref{eq:vis} with $t = \alpha_i$, we obtain
	\[q^{(k)}_{\alpha,\beta}v^{(i)}_{s} = q^{(k)}_{\alpha,\beta}v^{(i)}_{(i,\alpha_i),(i,s\alpha_i)} = 0.\]
	This proves the second case of the lemma (to obtain the result for $v^{(i)}_sq^{(k)}_{\alpha,\beta}$ simply apply the $^*$ operation to the above result), and the first case follows from this and the fact that $v^{(i)}_{s} + v^{(i)}_{-s} = 1$.
\end{proof}

\begin{lemma}\label{lem:qkcommute}
	Let $M \in \mathbb{Z}_2^{m \times n}$ and $b \in \mathbb{Z}_2^m$. Set $G = G(M,b)$ and let $u$ be the fundamental representation of $\Qut(G)$. Then $q^{(k)}_{\alpha,\beta}$ depends only on $k$ and the value of $\alpha \symmdiff \beta$ and therefore the entries of $q^{(k)}$ pairwise commute for each $k \in [m]$.
\end{lemma}
\begin{proof}
	Let $k \in [m]$ and $\alpha,\beta \in \pm 1^{S_k}_{b_k}$. We first show that $q^{(k)}_{\alpha,\beta}q^{(k)}_{\alpha',\beta'} = 0$ for any $\alpha', \beta' \in \pm 1^{S_k}_{b_k}$ such that $\alpha \symmdiff \beta \ne \alpha' \symmdiff \beta'$. In this case, there exists $i \in S_k$ such that $\alpha_i\beta_i \ne \alpha'_i\beta'_i$. We then have that
	\[q^{(k)}_{\alpha,\beta}q^{(k)}_{\alpha',\beta'} = q^{(k)}_{\alpha,\beta}(v^{(i)}_1 + v^{(i)}_{-1})q^{(k)}_{\alpha',\beta'} = q^{(k)}_{\alpha,\beta} v^{(i)}_1 q^{(k)}_{\alpha',\beta'} + q^{(k)}_{\alpha,\beta} v^{(i)}_{-1} q^{(k)}_{\alpha',\beta'}.\]
	But since $\alpha_i\beta_i \ne \alpha'_i\beta'_i$, each of the terms in the last sum is zero by \autoref{lem:qvprod}. Thus we have proven our claim.
	
	Now suppose that $\alpha,\beta, \tilde{\alpha}, \tilde{\beta} \in \pm 1^{S_k}_{b_k}$ are such that $\alpha \symmdiff \beta = \tilde{\alpha} \symmdiff \tilde{\beta}$. Note then that $q^{(k)}_{\alpha,\beta}q^{(k)}_{\tilde{\alpha},\beta'} = 0$ for any $\beta' \ne \tilde{\beta}$ by the above. Therefore,
	\[q^{(k)}_{\alpha,\beta} = q^{(k)}_{\alpha,\beta}\left(\sum_{\beta' \in \pm 1^{S_k}_{b_k}} q^{(k)}_{\tilde{\alpha},\beta'}\right) = q^{(k)}_{\alpha,\beta}q^{(k)}_{\tilde{\alpha},\tilde{\beta}} = \left(\sum_{\beta' \in \pm 1^{S_k}_{b_k}} q^{(k)}_{\alpha,\beta'}\right)q^{(k)}_{\tilde{\alpha},\tilde{\beta}} = q^{(k)}_{\tilde{\alpha},\tilde{\beta}},\]
	as desired. It follows from this that the magic unitary $q^{(k)}$ has the same set of entries in every row (just in different order) and therefore they all commute.
\end{proof}

\begin{definition}\label{def:vdelta}
	Since $q^{(k)}_{\alpha,\beta}$ depends only on the value of $\alpha \symmdiff \beta \in \pm 1^{S_k}_0$, for each $\delta \in \pm 1^{S_k}_0$ we define $q^{(k)}_\delta$ to be equal to $q^{(k)}_{\alpha,\beta}$ for some $\alpha,\beta \in \pm 1^{S_k}_{b_k}$ such that $\alpha \symmdiff \beta = \delta$ (note that such $\alpha$ and $\beta$ always exist, since we may take arbitrary $\alpha$ and let $\beta = \alpha \symmdiff \delta$).
\end{definition}

\begin{remark}
	Note that the set $\{q^{(k)}_\delta : \delta \in \pm 1^{S_k}_0\}$ is precisely the set of elements in every row/column of $q^{(k)}$, and thus
	\begin{equation}\label{eq:qdeltaid}
		\sum_{\delta \in \pm 1^{S_k}_0} q^{(k)}_\delta = 1.
	\end{equation}
\end{remark}

\begin{lemma}
	\label{lem:2gens}
	Let $M \in \mathbb{Z}_2^{m \times n}$ and $b \in \mathbb{Z}_2^m$ such that $|T_i|\geq 1$ for all $i\in [n]$. Set $G = G(M,b)$ and let $u$ be the fundamental representation of $\Qut(G)$. For any $k \in [m]$, $i \in S_k$, 
	\begin{equation}\label{eq:vfromq}
		v^{(i)}_s = \sum_{\delta \in \pm 1^{S_k}_{0}:\delta_i = s} q^{(k)}_{\delta},
	\end{equation}
	where $q^{(k)}_\delta$ is as defined in \autoref{def:vdelta}. It follows that $v^{(i)}_s$ and $v^{(j)}_t$ commute for all $i,j \in S_k$ and $s,t \in \{\pm 1\}$. Furthermore, for any $\delta \in \pm 1^{S_k}_0$,
	\begin{equation}\label{eq:qfromv}
		q^{(k)}_\delta = \prod_{i \in S_k}v^{(i)}_{\delta_i}.
	\end{equation}
	Therefore, the $C^*$-algebra $C(\Qut(G))$ is generated by each of the sets $\{q^{(k)}_\delta : k \in [m], \ \delta \in \pm 1^{S_k}_0\}$ and $\{v^{(i)}_s : i \in [n], \ s \in \{\pm 1\}\}$.
\end{lemma}
\begin{proof}
	Let $k \in [m]$, $i \in S_k$, and $s \in \{\pm 1\}$. If $\delta' \in \pm 1^{S_k}_0$ is such that $\delta'_i \ne s$, then for any $\alpha,\beta \in \pm 1^{S_k}_{b_k}$ such that $\alpha \symmdiff \beta = \delta'$, we have that $\alpha_i\beta_i \ne s$ and thus by \autoref{lem:qvprod} we have that
	\[v^{(i)}_s q^{(k)}_{\delta'} = v^{(i)}_s q^{(k)}_{\alpha,\beta} = 0.\]
	Therefore,
	\[v^{(i)}_s \sum_{\delta \in \pm 1^{S_k}_0 : \delta_i \ne s} q^{(k)}_\delta = 0.\]
	Applying \autoref{lem:resofid} to
	\[a_1 = v^{(i)}_1, \ a_2 = v^{(i)}_{-1}, \quad b_1 = \sum_{\delta \in \pm 1^{S_k}_0 : \delta_i = 1} q^{(k)}_\delta, \quad b_2 = \sum_{\delta \in \pm 1^{S_k}_0 : \delta_i = -1} q^{(k)}_\delta,\]
	we obtain Equation~\eqref{eq:vfromq}.
	We remark that this shows that the summation on the righthand side of Equation~\eqref{eq:vfromq} does not depend on the particular $k \in [m]$ such that $i \in S_k$.
	
	Now suppose that $i,j \in S_k$ and $s,t \in \{\pm 1\}$. By the above we have that both $v^{(i)}_s$ and $v^{(j)}_t$ are the sums of terms contained in the set $\{q^{(k)}_\delta : \delta \in \pm 1^{S_k}_0\}$. Since all terms in this set pairwise commute, so do $v^{(i)}_s$ and $v^{(j)}_t$. This means that the product on the righthand side of Equation~\eqref{eq:qfromv} is well-defined.
	
	Let $\delta \in \pm 1^{S_k}_0$ and pick $\alpha, \beta \in \pm 1^{S_k}_{b_k}$ so that $\delta = \alpha \symmdiff \beta$. From \autoref{lem:qvprod}, we have that
	\[q^{(k)}_\delta v^{(i)}_{\delta_i} = q^{(k)}_{\alpha, \beta} v^{(i)}_{\delta_i} = q^{(k)}_{\alpha, \beta} = q^{(k)}_{\delta},\]
	for all $i \in S_k$. From this it follows that
	\[q^{(k)}_\delta \prod_{i \in S_k}v^{(i)}_{\delta_i} = q^{(k)}_\delta.\]
	Now suppose that $\delta, \delta' \in \pm 1^{S_k}_0$ are not equal. Then there exists $j \in S_k$ such that $\delta_j \ne \delta'_j$. It follows that
	\[\left(\prod_{i \in S_k}v^{(i)}_{\delta_i}\right) \left(\prod_{i \in S_k}v^{(i)}_{\delta'_i}\right) = \left(\prod_{i \in S_k \setminus \{j\}}v^{(i)}_{\delta_i}\right) v^{(j)}_{\delta_j}v^{(j)}_{\delta'_j}\left(\prod_{i \in S_k \setminus \{j\}}v^{(i)}_{\delta'_i}\right) = 0.\]
	The above two equations allow us to apply \autoref{cor:resofid} and conclude that Equation~\eqref{eq:qfromv} holds.
	
	Lastly, it is clear that union of the two sets of generators from the lemma statement generate $C(\Qut(G))$, and by Equations~\eqref{eq:vfromq} and~\eqref{eq:qfromv} and the fact that $|T_i| \geq 1$ for all $i$, we see that each set individually generates the other. Therefore each set generates the full algebra $C(\Qut(G))$. Note that if $|T_i| = 0$ then the element $v^{(i)}_s$ could not be generated by the elements $\{q^{(k)}_\delta : k \in [m], \ \delta \in \pm 1^{S_k}_0\}$.
\end{proof}

In order to show that $\Qut(G(M,b))$ is isomorphic to the dual of the solution group $\Gamma(M,0)$, we will show that $\Qut(G(M,b))$ is isomorphic to $\Qut(\widehat{G}(M,b))$ in the theorem below and then apply~\autoref{thm:qaut2cgamma}.

\begin{theorem}
	\label{thm:qautG2qautGhat}
	Let $M \in \mathbb{Z}_2^{m \times n}$ and $b \in \mathbb{Z}_2^m$ such that $|T_i|\geq 1$ for all $i\in [n]$. Set $G = G(M,b)$ and $\widehat{G} = \widehat{G}(M,b)$. Let $\Delta_G$ and $\Delta_{\widehat{G}}$ denote the comultiplications of $\Qut(G)$ and $\Qut(\widehat{G})$ respectively. Then there exists an isomorphism $\varphi : C(\Qut(G)) \to C(\Qut(\widehat{G}))$ such that $\Delta_{\widehat{G}} \circ \varphi = (\varphi \otimes \varphi) \circ \Delta_G$. In other words, $\Qut(G)$ and $\Qut(\widehat{G})$ are isomorphic as compact quantum groups.
\end{theorem}

\begin{proof}
	The proof is structured as follows: we use the universal properties of $C(\Qut(G))$ and $C(\Qut(\widehat{G}))$ respectively to first obtain a surjective $*$-homomorphism $\varphi_1: C(\Qut(\widehat{G})) \to C(\Qut(G))$, and then a surjective $*$-homomorphism $\varphi_2: C(\Qut(G)) \to C(\Qut(\widehat{G}))$. We will then show that $\varphi_1$ and $\varphi_2$ are inverses of each other thus proving that they are in fact isomorphisms. Lastly, we will prove that $\varphi_2$ intertwines the comultiplications as described in the theorem statement.\\
	
	\noindent\emph{Step 1: Construction of a $*$-homomorphism $\varphi_1: C(\Qut(\widehat{G})) \to C(\Qut(G))$.}\\
	We let $u$ and $\hat{u}$ denote the fundamental representations of $\Qut(G)$ and $\Qut(\widehat{G})$ respectively. We also use $q^{(k)}_{\alpha,\beta}$ and $\hat{q}^{(k)}_{\alpha,\beta}$ to denote $u_{(k,\alpha),(k,\beta)}$ and $\hat{u}_{(k,\alpha),(k,\beta)}$ respectively. We will show that the elements $q^{(k)}_{\alpha,\beta}$ satisfy the same relations as the $\hat{q}^{(k)}_{\alpha,\beta}$, thus proving the existence of the $*$-homomorphism $\varphi_1$ that maps $\hat{q}^{(k)}_{\alpha,\beta}$ to $q^{(k)}_{\alpha,\beta}$.
	
	We must show that $q^{(k)}_{\alpha,\beta}q^{(l)}_{\alpha',\beta'} = 0$ whenever the edge in $\widehat{G}$ between $(k,\alpha)$ and $(l,\alpha')$ is a different color than the edge between $(k,\beta)$ and $(l,\beta')$, or when one is a non-edge and the other is an edge. If $k \ne l$, then the latter case never occurs by definition of the graph $\widehat{G}$. If $k = l$, then it is possible that one pair is a non-edge and one pair is an edge, but in this case the non-edge pair must be equal (and the edge pair will not be equal), and thus the required orthogonality is immediate. Thus we may assume that they are both edges and thus $S_l \cap S_k \ne \varnothing$. Since the edges have different colors, we have that $\alpha \symmdiff \alpha' \ne \beta \symmdiff \beta'$, and thus there exists $i \in S_l \cap S_k$ such that $\alpha_i\alpha'_i \ne \beta_i\beta'_i$ which is equivalent to $\alpha_i\beta_i \ne \alpha'_i \beta'_i$. Therefore,
	\[q^{(k)}_{\alpha,\beta}q^{(l)}_{\alpha',\beta'} = q^{(k)}_{\alpha,\beta} \left(v^{(i)}_1 + v^{(i)}_{-1}\right) q^{(l)}_{\alpha',\beta'} = q^{(k)}_{\alpha,\beta} v^{(i)}_1 q^{(l)}_{\alpha',\beta'} + q^{(k)}_{\alpha,\beta} v^{(i)}_{-1} q^{(l)}_{\alpha',\beta'} = 0 + 0\]
	by \autoref{lem:qvprod}. By the universality property of $C(\Qut(\widehat{G}))$, there is a $*$-homomorphism $\varphi_1: C(\Qut(\widehat{G})) \to C(\Qut(G))$ such that $\varphi_1(\hat{q}^{(k)}_{\alpha,\beta}) = q^{(k)}_{\alpha,\beta}$ for all $k \in [m]$ and $\alpha,\beta \in \pm 1^{S_k}_{b_k}$. Since the elements $q^{(k)}_{\alpha,\beta}$ for $k \in [m]$ and $\alpha,\beta \in \pm 1^{S_k}_{b_k}$ generate $C(\Qut(G))$ by \autoref{lem:2gens}, the $*$-homomorphism $\varphi_1$ is surjective.\\
	
	\noindent\emph{Step 2: Construction of a $*$-homomorphism $\varphi_2: C(\Qut(G)) \to C(\Qut(\widehat{G}))$.}\\
	We construct elements that generate $C(\Qut(\widehat{G}))$ and satisfy the same relations as the entries of the fundamental representation of $\Qut(G)$ and thus by the universality of $C(\Qut(G))$ there is a surjective $*$-homomorphism $\varphi_2: C(\Qut(G)) \to C(\Qut(\widehat{G}))$. We will have that $\varphi_2(q^{(k)}_{\alpha,\beta}) = \hat{q}^{(k)}_{\alpha,\beta}$, and we will define elements $\hat{v}^{(i)}_{s,t}$ of $C(\Qut(\widehat{G}))$ such that $\varphi_2(v^{(i)}_{s,t}) = \hat{v}^{(i)}_{s,t}$.
	
	Analogously to this work, in~\cite{RSsolution} it was shown that $\hat{q}^{(k)}_{\alpha,\beta}$ (denoted $u^{(k)}_{\alpha,\beta}$ in~\cite{RSsolution}) depends only on $\alpha \symmdiff \beta$, and therefore we can define $\hat{q}^{(k)}_\delta$ for all $\delta \in \pm 1^{S_k}_0$ analogously to $q^{(k)}_\delta$. Now for $k \in [m]$, $i \in S_k$, and $s \in \{\pm 1\}$, define 
	\[\hat{v}^{(i,k)}_s = \sum_{\delta \in \pm 1^{S_k}_0: \delta_i = s} \hat{q}^{(k)}_\delta.\]
	It is easy to see that $\hat{v}^{(i,k)}_s = (\hat{v}^{(i,k)}_s)^* = (\hat{v}^{(i,k)}_s)^2$ and that $\hat{v}^{(i,k)}_1 + \hat{v}^{(i,k)}_{-1} = 1$. In~\cite{RSsolution}, it was shown that the element $y_i^{(k)} := \hat{v}^{(i,k)}_1 - \hat{v}^{(i,k)}_{-1}$ does not depend on the $k \in [m]$ such that $i \in S_k$. Since $\hat{v}^{(i,k)}_s = \frac{1}{2}(1 + sy^{(k)}_i)$, it follows that $\hat{v}^{(i,k)}_s$ also does not depend on $k$ and so we can simply write $\hat{v}^{(i)}_s$. Finally, we define $\hat{v}^{(i)}_{s,t} = \hat{v}^{(i)}_{st}$ for $s,t \in \{\pm 1\}$, and let $\hat{v}^{(i)}$ denote the magic unitary $(\hat{v}^{(i)}_{s,t})_{s,t \in \{\pm 1\}}$.
	
	It is easy to see that
	\[\hat{\hat{u}} := \left(\bigoplus_{i \in [n]} \hat{v}^{(i)}\right) \oplus \left(\bigoplus_{k \in [m]} \hat{q}^{(k)}\right)\]
	is a magic unitary indexed by the vertices of $G$. Therefore, it is only left to show that its entries satisfy the same relations as the entries of the fundamental representation of $\Qut(G)$. The relations corresponding to the vertex colors are satisfied simply due to the block form of $\hat{\hat{u}}$. Therefore, it suffices to show that $\hat{\hat{u}}_{a,b}\hat{\hat{u}}_{a',b'} = 0$ whenever $a \sim a'$ and $b \nsim b'$. Since $a \sim a'$, we may assume without loss of generality that $a = (i,s)$ for some $i \in [n]$ and $s \in \{\pm 1\}$, and $a' = (k,\alpha)$ for some $k \in [m]$ such that $i \in S_k$ and $\alpha \in \pm 1^{S_k}_{b_k}$. If $\hat{\hat{u}}_{a,b} = 0$, then we are done, so we may assume that $b = (i,t)$ for some $t \in \{\pm 1\}$. Similarly, we may assume that $b' = (k,\beta)$ for some $\beta \in \pm 1^{S_k}_{b_k}$. Thus we must show that $\hat{v}^{(i)}_{st}\hat{q}^{(k)}_{\alpha \symmdiff \beta} = 0$. Since $a \sim a'$, we have that $\alpha_i = s$, and since $b \nsim b'$ we have that $\beta_i \ne t$. Letting $\delta = \alpha \symmdiff \beta$, we see that $\delta_i \ne st$. Using the definition of $\hat{v}^{(i)}_{st}$, we have that
	\begin{equation}\label{eq:vq2qq}
	\hat{v}^{(i)}_{st}\hat{q}^{(k)}_{\delta} = \left(\sum_{\delta' \in \pm 1^{S_k}_0: \delta'_i = st} \hat{q}^{(k)}_{\delta'}\right)\hat{q}^{(k)}_{\delta}.
	\end{equation}
	Since $\hat{q}^{(k)}_{\delta'}$ for $\delta' \in \pm 1^{S_k}_0$ are precisely the entries in any row of the magic unitary $\hat{q}^{(k)}$, we have that $\hat{q}^{(k)}_{\delta'}\hat{q}^{(k)}_{\delta} = 0$ unless $\delta'=\delta$. However, since $\delta_i \ne st$ the expression in~\eqref{eq:vq2qq} is equal to zero as desired.
	
	By the above and the universality of $C(\Qut(G))$, there exists a surjective $*$-homomorphism $\varphi_2:C(\Qut(G)) \to C(\Qut(\widehat{G}))$ such that $\varphi_2(q^{(k)}_{\alpha,\beta}) = \hat{q}^{(k)}_{\alpha,\beta}$ for all $k \in [m]$ and $\alpha,\beta \in \pm 1^{S_k}_{b_k}$, and $\varphi_2(v^{(i)}_s) = \hat{v}^{(i)}_s$ for all $i \in [n]$ and $s \in \{\pm 1\}$.\\
	
	\noindent\emph{Step 3: Showing that $\varphi_1$ and $\varphi_2$ are inverses of each other.}\\
	Since $\varphi_1$ and $\varphi_2$ are $*$-homomorphisms, it suffices to show that $\varphi_1 \circ \varphi_2 : C(\Qut(G)) \to C(\Qut(G))$ acts as the identity on a generating set of $C(\Qut(G))$. Thus by \autoref{lem:2gens} it suffices to show that $\varphi_1(\varphi_2(q^{(k)}_{\alpha,\beta})) = q^{(k)}_{\alpha,\beta}.$ However, this is immediate from the above since we have shown that $\varphi_2(q^{(k)}_{\alpha,\beta}) = \hat{q}^{(k)}_{\alpha,\beta}$ in Step 2, and that $\varphi_1(\hat{q}^{(k)}_{\alpha,\beta}) = q^{(k)}_{\alpha,\beta}$ in Step 1. Similarly, $\varphi_2(\varphi_1(\hat{q}^{(k)}_{\alpha,\beta})) = \hat{q}^{(k)}_{\alpha,\beta}$ and therefore $\varphi_2 \circ \varphi_1$ is the identity on $C(\Qut(\widehat{G}))$. Thus we have shown that $\varphi_1$ and $\varphi_2$ are inverses of each other and therefore they are both $*$-isomorphisms.\\
	
	\noindent\emph{Step 4: Showing that $\Delta_G \circ \varphi_1 = (\varphi_1 \otimes \varphi_1) \circ \Delta_{\widehat{G}}$.}\\
	Again, it suffices to prove the equality on a generating set of $C(\Qut(\widehat{G}))$. In this case we will show that these two functions are equal on $\hat{q}^{(k)}_{\alpha,\beta}$ for all $k \in [m]$ and $\alpha, \beta \in \pm 1^{S_k}_{b_k}$. Due to the block diagonal structure of the fundamental representation $\hat{u}$ of $\Qut(\widehat{G})$, it is easy to see that
 \[\Delta_{\widehat{G}}(\hat{q}^{(k)}_{\alpha,\beta}) = \sum_{\gamma \in \pm 1^{S_k}_{b_k}} \hat{q}^{(k)}_{\alpha,\gamma} \otimes \hat{q}^{(k)}_{\gamma,\beta}.\]
 Therefore,
 \[(\varphi_1 \otimes \varphi_1) \circ \Delta_{\widehat{G}}(\hat{q}^{(k)}_{\alpha,\beta}) = \sum_{\gamma \in \pm 1^{S_k}_{b_k}} q^{(k)}_{\alpha,\gamma} \otimes q^{(k)}_{\gamma,\beta}.\]
On the other hand, $\varphi(\hat{q}^{(k)}_{\alpha,\beta}) = q^{(k)}_{\alpha,\beta}$, and by the block structure of the fundamental representation $u$ of $\Qut(G)$, we have that 
\[\Delta_G(q^{(k)}_{\alpha,\beta}) = \sum_{\gamma \in \pm 1^{S_k}_{b_k}} q^{(k)}_{\alpha,\gamma} \otimes q^{(k)}_{\gamma,\beta}.\]
Therefore, $\Delta_G \circ \varphi_1(\hat{q}^{(k)}_{\alpha,\beta}) = (\varphi_1 \otimes \varphi_1) \circ \Delta_{\widehat{G}}(\hat{q}^{(k)}_{\alpha,\beta})$ for all $k \in [m]$ and $\alpha,\beta \in \pm 1^{S_k}_{b_k}$, and thus $\Delta_G \circ \varphi_1 = (\varphi_1 \otimes \varphi_1) \circ \Delta_{\widehat{G}}$ as desired.
\end{proof}

As a corollary of~\autoref{thm:qaut2cgamma} and~\autoref{thm:qautG2qautGhat}, we obtain the main result of this section. We denote by $G \sqcup G'$ the disjoint union of graphs $G$ and $G'$.
Note that we obtain a vertex-colored graph whose quantum automorphism group is the dual of a homogeneous solution group even if the solution group has variables that do not appear in any equation.

\begin{corollary}\label{cor:vertexcolored}
	Let $M \in \mathbb{Z}_2^{m \times n}$ and $b \in \mathbb{Z}_2^m$. Then $\Qut(G(M,b))$ is isomorphic to the dual of the solution group $\Gamma(M,0)$ as a compact quantum group.
\end{corollary}

\begin{proof}
	By reordering, we may assume that there is an $n_1 \leq n$ such that $|T_i|\geq 1$ for all $i\leq n_1$ and $|T_i|=0$ for all $k$ for all $i > n_1$. Let $s=n-n_1$. It holds
	\begin{align*}
		M=\begin{pmatrix} M' & 0_{m,s}\end{pmatrix},
	\end{align*}
	where $M'\in \mathbb{Z}_2^{m \times n_1}$ consists of the first $n_1$ columns of $M$ and $0_{m,s}\in \mathbb{Z}_2^{m \times s}$ is the zero matrix. By \autoref{def:graph}, we obtain
	\begin{align*}
		G(M,b)\cong G(M',b) \sqcup \left(\bigsqcup_{i > n_1} \overline{K_2}^{(i,0)}\right),
	\end{align*}
	where $\overline{K_2}^{(i,0)}$ denotes the two isolated vertices $(i,1)$ and $(i,-1)$ with color $(i,0)$. Therefore, we get 
	\begin{align}
		\Qut(G(M,b))&\cong \Qut\left(G(M',b) \sqcup \left(\bigsqcup_{i > n_1} \overline{K_2}^{(i,0)}\right)\right)\nonumber\\
		&\cong \Qut(G(M',b)) \ast (\bigast_{i > n_1} \Qut(\overline{K_2}^{(i,0)}))\label{eq:freeprod},
	\end{align}
	by \cite[Lemma 5.4]{DKRSZ}, since the graphs $G(M',b)$ and $\overline{K_2}^{(i,0)}$ are not quantum isomorphic as vertex-colored graphs. By \autoref{thm:qautG2qautGhat}, we obtain
	\begin{align*}
		\Qut(G(M',b)) \ast& \left(\bigast_{i > n_1} \Qut(\overline{K_2}^{(i,0)})\right)\\&\cong \Qut(\hat{G}(M',b)) \ast \left(\bigast_{i > n_1} \Qut(\overline{K_2}^{(i,0)})\right),
	\end{align*}
	since for all $i\in [n_1]$, we have $M_{ki}'=1$ for some $k$ by construction of $M'$. It holds $\Qut(\overline{K_2}^{(i,0)})\cong \hat{\Z}_2$ and therefore $\bigast_{i > n_1} \Qut(\overline{K_2}^{(i,0)})\cong \widehat{\Z_2^{\ast s}}$. Furthermore, we have $ \Qut(\hat{G}(M',b)) \cong \hat{\Gamma}(M',0)$ by \autoref{thm:qaut2cgamma}. Summarizing, we obtain
	{\reqnomode 
	\begin{align*}
		\Qut(G(M,b))\cong \hat{\Gamma}(M',0)\ast \widehat{\Z_2^{\ast s}}\cong \hat{\Gamma}(M,0). \tag*{\qedhere}
	\end{align*}
	} 
\end{proof}

\section{Decoloring arbitrary vertex-colored graphs while preserving quantum isomorphism and quantum automorphism groups}
\label{sec:decolor}

\def\X{G}

In this section, we give a decoloring procedure
\[ \decolor : \{\text{vertex-colored graphs}\} \to \{\text{uncolored, connected graphs}\} \]
that preserves quantum automorphism groups as well as quantum isomorphism between different graphs.
This generalizes the decoloring procedure from \cite[Corollary 4.14]{RSsolution}, which only applies to graphs with minimum degree $\geq 3$, and even then does not always preserve the quantum automorphism group.%
	\footnote{The decoloring procedure of \cite{RSsolution} adds new constraints on the generators of the fundamental representation, so the quantum automorphism group is only preserved if these constraints were already met in the original graph.}
Moreover, the decoloring procedure from \cite{RSsolution} was not shown to preserve quantum isomorphism in general.
On the other hand, the decoloring procedure of \cite{RSsolution} is more general in another respect, because it applies to vertex-- and edge-colored graphs (instead of just vertex-colored graphs).

A precise statement of the main result of this section can be found in \autoref{thm:decolor} and \autoref{cor:decolor} below.

\subsection{The quantum isomorphism \texorpdfstring{$C^*$\nobreakdash-algebra}{C*-algebra} of two graphs}
To formulate our result, we expand our view from quantum automorphism groups to quantum isomorphism $C^*$\nobreakdash-algebras.

\begin{definition}[{\cite[Remark 2.4]{Brannan-bigalois}, \cite[Definition 6.3]{RSsolution}}]
	\label{def:QIso}
	Let $\X$ and $\X'$ be vertex-colored graphs whose colorings $c : V(\X) \to C$, $c' : V(\X') \to C$ take values in a common set of colors $C$.
	The \emph{quantum isomorphism $C^*$\nobreakdash-algebra of $\X$ and $\X'$}, denoted $\QIso(\X,\X')$, is the universal $C^*$\nobreakdash-algebra generated by $u = (u_{ij})_{i \in V(\X), j \in V(\X')}$ satisfying:
	\begin{align*}
		&\bullet \ u_{ij}^2 = u_{ij}^* = u_{ij} &&\quad \text{for all $i \in V(\X)$, $j \in V(\X')$}; \\[.5ex]
		&\bullet \ \sum_{j \in V(\X')} u_{ij} = 1 &&\quad \text{for all $i \in V(\X)$}; \\
		&\bullet \ \sum_{i \in V(\X)} u_{ij} = 1 &&\quad \text{for all $j \in V(\X')$}; \\[.5ex]
		&\bullet \ A_\X u = u A_{\X'} \, ; \\[.5ex]
		&\bullet \ u_{ij} = 0 &&\quad \text{whenever $c(i) \neq c'(j)$}.
	\end{align*}
	We refer to $u$ as the \emph{generating matrix of $\QIso(\X,\X')$}.
	
	For vertex-colored graphs $\X,\X',\tilde \X,\tilde \X'$, we write $\QIso(\X,\X') = \QIso(\tilde \X,\tilde \X')$ if and only if $V(\X) = V(\tilde \X)$ and $V(\X') = V(\tilde \X')$ and $u = \tilde u$, where $u$ and $\tilde u$ are the generating matrices of $\QIso(\X,\X')$ and $\QIso(\tilde \X,\tilde \X')$, respectively.%
		\footnote{More formally, by $u = \tilde u$ we mean that there are $C^*$\nobreakdash-algebra isomorphisms $\varphi : \QIso(\X,\X') \to \QIso(\tilde \X,\tilde \X')$, $\psi : \QIso(\tilde \X,\tilde \X') \to \QIso(\X,\X')$ such that $\varphi(u_{ij}) = \tilde u_{ij}$ and $\psi(\tilde u_{ij}) = u_{ij}$ for all $i \in V(\X) = V(\tilde \X)$, $j \in V(\X') = V(\tilde \X')$.}
	General $C^*$\nobreakdash-algebra isomorphisms (that do not necessarily preserve the generating matrix) are denoted $\QIso(\X,\X') \cong \QIso(\tilde \X,\tilde \X')$.
\end{definition}

\begin{remark}
	\label{rmk:iso-qiso-algebra}
	The relations on $u$ are exactly the conditions of quantum isomorphism (see \autoref{def:qiso}), so it follows from universality of the $C^*$\nobreakdash-algebra $\QIso(\X,\X')$ that $\X \cong_q \X'$ if and only if $\QIso(\X,\X') \neq 0$.
	By a similar argument, $\X$ and $\X'$ are classically isomorphic if and only if $\QIso(\X,\X')$ has a non-trivial one-dimensional $*$\nobreakdash-representation.
\end{remark}

\begin{remark}
	\label{rmk:qiso-vertex-counts}
	If $\QIso(\X,\X') \neq 0$, then summing over the row and column sums of the generating matrix $u$ reveals that $|V(\X)| = |V(\X')|$, so quantum isomorphic graphs must have the same number of vertices.
	Likewise, since $u_{ij} = 0$ whenever $c(i) \neq c'(j)$, the rows and columns of the submatrix of $u$ indexed by all $i \in V(\X), j \in V(\X')$ with $c(i) = c'(j) = r$ sum up to $1$, for every fixed color $r \in C$.
	So again, summing over these row and column sums reveals that quantum isomorphic graphs $\X \cong_q \X'$ must have the same number of vertices of color $r$, for every color $r \in C$.
\end{remark}

\begin{definition}	
	Given vertex-colored graphs $\X,\X',\X''$, it is routinely verified that there exists a (well-defined) $*$\nobreakdash-homomorphism
	\begin{align*}
		&\Delta_{\X,\X',\X''} : \QIso(\X,\X'') \to \QIso(\X,\X') \otimes \QIso(\X',\X'')
	\end{align*}
	such that
	\begin{align*}
		&\Delta_{\X,\X',\X''}(u_{ik}) = \sum_{j \in V(\X')} v_{ij} \otimes w_{jk} \qquad \text{for all $i \in V(\X)$, $k \in V(\X'')$},
	\end{align*}
	where $u,v,w$ denote the respective generating matrices.
	We call $\Delta_{\X,\X',\X''}$ the \emph{co-composition}.
	When $\X$ and $\X''$ are clear from context, we write $\Delta_{\X'}$ instead of $\Delta_{\X,\X',\X''}$.
	Note that the co-composition $\Delta_{\X,\X,\X}$ coincides with the comultiplication of $\Qut(\X)$.
\end{definition}

\begin{remark}
	\label{rmk:groupoid}
	The quantum isomorphism $C^*$\nobreakdash-algebras $\QIso(\X,\X')$ do not form a quantum group when $\X \neq \X'$.
	However, they do form a ``cogroupoid'' in the sense of \cite{Bichon-cogroupoid}; this was proved in \cite{Brannan-bigalois}.%
		\footnote{In \cite{Brannan-bigalois}, this was shown at the level of $*$\nobreakdash-algebras (instead of $C^*$\nobreakdash-algebras), but it is not hard to see that the same holds for the $C^*$\nobreakdash-algebras.
		This boils down to showing that the maps $\varepsilon_{\X} : \Qut(\X) \to \C$, $u_{ij} \mapsto \delta_{ij}$ and $S_{\X,\X'} : \QIso(\X,\X') \to \QIso(\X',\X)$, $u_{ij} \mapsto u_{ji}'$ satisfy the properties of the counit and coinverse.}
	
	The quantum isomorphism $C^*$\nobreakdash-algebra $\QIso(\X,\X')$ was originally defined in quantum information theory as the $C^*$\nobreakdash-algebra of the $(\X,\X')$-isomorphism game; see in particular \cite[Remark 2.4]{Brannan-bigalois}.
	Connections between quantum automorphism groups and quantum isomorphism have been studied before; see in particular \cite[Theorem 4.5]{qperms} and \cite[Theorem 5.3]{quantum-Sabidussi}.
\end{remark}

\subsection{The decoloring procedure}

We now present the decoloring procedure and state the main result of this section.

\begin{definition}
	\label{def:decolor}
	Let $\X$ be a vertex-colored graph with colors in the set $\N_1 \times \{0\} = \{(1,0),(2,0),\ldots\}$.
	The \emph{decoloring of $\X$}, denoted $\decolor(\X)$, is the uncolored graph obtained from $\X$ in the following way:
	\begin{itemize}[leftmargin=21mm,labelwidth=\widthof{Notation.}]
		\item[Notation.] Let $C_\X \subseteq \N_1 \times \{0\}$ denote the set of colors that occur at least once in $\X$.
		\item[Step 1.] For every $(c,0) \in C_\X$, add $c + 4$ vertices $x_{(c,1)},\ldots,x_{(c,c+4)}$ with respective colors $(c,1),\ldots,(c,{c+4})$.
		Furthermore, add a vertex $x_{(0,0)}$ with color $(0,0)$.
		\item[Step 2.] For every $(c,0) \in C_\X$, connect the vertices $x_{(c,1)},\ldots,x_{(c,c+4)}$ to form a path $x_{(c,1)}x_{(c,2)}\cdots x_{(c,c+4)}$.
		Furthermore, connect $x_{(c,1)}$ to all vertices of color $(c,0)$, and connect $x_{(c,2)}$ to $x_{(0,0)}$.
		\item[Step 3.] Forget the colors, and let $\decolor(\X)$ be the (uncolored) graph thus obtained.
	\end{itemize}
	The steps of this decoloring procedure are illustrated in \autoref{fig:decolor} ($\X_2$ denotes the graph after Step 1 and Step 2).
	Note that $\decolor(\X)$ is always connected, since from every vertex there is a path to $x_{(0,0)}$.
	\begin{figure}[h!t]
		\centering
		\definecolor{mycolor0}{hsb}{0, .9, .85}
		\definecolor{mycolor1}{hsb}{.1, .95, 1}
		\definecolor{mycolor2}{hsb}{.18, 1, 1}
		\definecolor{mycolor3}{hsb}{.27, .75, 1}
		\definecolor{mycolor4}{hsb}{.45, .9, .7}
		\definecolor{mycolor5}{hsb}{.5, 1, 1}
		\definecolor{mycolor6}{hsb}{.58, .7, .7}
		\definecolor{mycolor7}{hsb}{.77, 1, .9}
		\definecolor{mycolor8}{hsb}{.88, .8, 1}
		\begin{tikzpicture}[vertex/.style={circle,fill,inner sep=1.42pt},scale=.87]
			\def\numverts{16}
			\def\numcolors{9}
			\begin{scope}[xshift=-1.25mm,yshift=3.7cm,scale=1.2]
				\foreach \x/\clr in {0/0, 1/0, 2/0, 3/1, 4/2, 5/2, 6/3, 7/4, 8/4, 9/4, 10/5, 11/5, 12/6, 13/7, 14/7, 15/8} {
					\pgfmathsetmacro{\hoek}{90 - 360 * \x / \numverts}
					\def\curcolor{mycolor\clr}
					\pgfmathtruncatemacro{\lbl}{\clr + 1}
					\pgfmathsetmacro{\mylspace}{4*max(cos(\hoek), 0)}
					\pgfmathsetmacro{\myrspace}{4*max(-cos(\hoek), 0)}
					\draw[fill=\curcolor,draw=\curcolor!75!black] (\hoek:1cm) node[vertex,draw,semithick] (v\x) {} ++(\hoek:2.2mm) node[scale=.5,\curcolor!45!black] {\hspace*{\mylspace mm}$(\lbl,0)$\hspace*{\myrspace mm}};
				}
				\draw (v1) -- (v0) to[bend right=20] (v2);
				\draw (v2) -- (v10) -- (v1) -- (v2);
				\draw (v5) -- (v14) (v8) to[bend left=15] (v5);
				\draw (v14) -- (v6) -- (v11) to[bend right=20] (v13);
				\draw (v12) -- (v7);
				\draw (v11) -- (v7) to[bend left=20] (v5);
				\draw (v7) -- (v6) -- (v5) -- (v13) -- (v6);
				\node[anchor=base] at (-1.3,-1.4) {$\X$};
			\end{scope}
			\draw[->,>=stealth'] (2.1,3.7) to node[above,scale=.7,yshift=2pt] {Steps 1 \&{} 2} ++(2,0);
			\begin{scope}[xshift=8.7cm,yshift=.5cm]
				\foreach \x/\clr in {0/0, 1/0, 2/0, 3/1, 4/2, 5/2, 6/3, 7/4, 8/4, 9/4, 10/5, 11/5, 12/6, 13/7, 14/7, 15/8} {
					\pgfmathsetmacro{\hoek}{90 - 360 * \x / \numverts}
					\def\curcolor{mycolor\clr}
					\draw[fill=\curcolor,draw=\curcolor!75!black] (\hoek:1cm) node[vertex,draw,semithick] (v\x) {};
				}
				\foreach \x/\clr in {1/0, 3/1, 4.5/2, 6/3, 8/4, 10.5/5, 12/6, 13.5/7, 15/8} {
					\pgfmathsetmacro{\hoek}{90 - 360 * \x / \numverts}
					\pgfmathsetmacro{\perphoek}{\hoek - 90}
					\def\curcolor{mycolor\clr}
					\pgfmathtruncatemacro{\lbl}{\clr + 1}
					\pgfmathtruncatemacro{\numnewverts}{\clr + 5}
					\foreach \y in {1,...,\numnewverts} {
						\pgfmathsetmacro{\nodedist}{1.4 + .25 * (\y - 1)}
						\pgfmathsetmacro{\fillfactor}{(70 - 6*\y)^2 / 70 + 12}
						\pgfmathsetmacro{\drawfactor}{(90 - 8*\y)^2 / 70 + 20}
						\pgfmathsetmacro{\textfactor}{(90 - 6*\y)^2 / 135 + 45}
						\draw (\hoek:\nodedist cm) node[vertex,draw,line width=.5pt,fill=\curcolor!\fillfactor,draw=\curcolor!75!black!\drawfactor] (x-\clr-\y) {} ++(\perphoek:.6mm) node[anchor=west,scale=.5,rotate=\perphoek,\curcolor!45!black!\textfactor] {${(\lbl,\y)}$};
						\pgfmathtruncatemacro{\yprev}{\y - 1}
						\ifnum\y>1 \draw (x-\clr-\y) -- (x-\clr-\yprev); \fi
					}
				}
				\foreach \x/\clr in {0/0, 1/0, 2/0, 3/1, 4/2, 5/2, 6/3, 7/4, 8/4, 9/4, 10/5, 11/5, 12/6, 13/7, 14/7, 15/8} {
					\draw (v\x) -- (x-\clr-1);
				}
				\draw (v1) -- (v0) to[bend right=20] (v2);
				\draw (v2) -- (v10) -- (v1) -- (v2);
				\draw (v5) -- (v14) (v8) to[bend left=15] (v5);
				\draw (v14) -- (v6) -- (v11) to[bend right=20] (v13);
				\draw (v12) -- (v7);
				\draw (v11) -- (v7) to[bend left=20] (v5);
				\draw (v7) -- (v6) -- (v5) -- (v13) -- (v6);
				\draw[fill=white,draw] (1.6,4.5) node[vertex,draw] (x00) {} ++(0,.2) node[scale=.5] {$(0,0)$};
				\foreach \x/\clr in {1/0, 3/1, 4.5/2, 6/3, 8/4, 10.5/5, 12/6, 13.5/7, 15/8} {
					\pgfmathsetmacro{\hoek}{90 - 360 * \x / \numverts}
					\begin{scope}[rotate=\hoek]
						\draw (x-\clr-2) ++(0,.25) coordinate (Lctrl-\clr-base);
						\foreach \y in {1,...,4} {
							\pgfmathsetmacro{\nodedist}{1.4 + .25 * (\clr + 4 + \y) + .3 - .12 * \y}
							\coordinate (Lctrl-\clr-\y) at (\nodedist,.3);
							\coordinate (Rctrl-\clr-\y) at (\nodedist,-.8);
						}
					\end{scope}
				}
				\draw[rounded corners] (x-0-2) -- (Lctrl-0-base) -- (x00);
				\draw[rounded corners] (x-1-2) -- (Lctrl-1-base) -- (Lctrl-1-1) -- (x00);
				\draw[rounded corners] (x-2-2) -- (Lctrl-2-base) -- (Rctrl-1-2) -- (Lctrl-1-2) -- (x00);
				\draw[rounded corners] (x-3-2) -- (Lctrl-3-base) -- (Rctrl-2-2) -- (Rctrl-1-3) -- (Lctrl-1-3) -- (x00);
				\draw[rounded corners] (x-4-2) -- (Lctrl-4-base) -- (Rctrl-3-2) -- (Rctrl-2-3) -- (Rctrl-1-4) -- (Lctrl-1-4) -- (x00);
				\draw[rounded corners] (x-5-2) -- (Lctrl-5-base) -- (Lctrl-5-1) -- (Lctrl-6-2) -- (Lctrl-7-3) -- (Lctrl-8-4) -- (Rctrl-8-4) -- (x00);
				\draw[rounded corners] (x-6-2) -- (Lctrl-6-base) -- (Lctrl-6-1) -- (Lctrl-7-2) -- (Lctrl-8-3) -- (Rctrl-8-3) -- (x00);
				\draw[rounded corners] (x-7-2) -- (Lctrl-7-base) -- (Lctrl-7-1) -- (Lctrl-8-2) -- (Rctrl-8-2) -- (x00);
				\draw[rounded corners] (x-8-2) -- (Lctrl-8-base) -- (Lctrl-8-1) -- (Rctrl-8-1) -- (x00);
				\node[anchor=base] at (-1.9,-3.2) {$\X_2$};
				\coordinate (G2_bottom) at (-.2,-4.4);
			\end{scope}
			\draw[->,>=stealth',rounded corners] (G2_bottom) |- ++(-1,-.5) to node[below,scale=.7,yshift=-2pt] {Step 3} ++(-1,0) -- ++(-1,0);
			\begin{scope}[xshift=1.5cm,yshift=-3.5cm,scale=.85,vertex/.style={circle,fill,inner sep=1pt}]
				\foreach \x/\clr in {0/0, 1/0, 2/0, 3/1, 4/2, 5/2, 6/3, 7/4, 8/4, 9/4, 10/5, 11/5, 12/6, 13/7, 14/7, 15/8} {
					\pgfmathsetmacro{\hoek}{90 - 360 * \x / \numverts}
					\draw[fill=black,draw=black] (\hoek:1cm) node[vertex,draw,semithick] (v\x) {};
				}
				\foreach \x/\clr in {1/0, 3/1, 4.5/2, 6/3, 8/4, 10.5/5, 12/6, 13.5/7, 15/8} {
					\pgfmathsetmacro{\hoek}{90 - 360 * \x / \numverts}
					\pgfmathtruncatemacro{\numnewverts}{\clr + 5}
					\foreach \y in {1,...,\numnewverts} {
						\pgfmathsetmacro{\nodedist}{1.4 + .25 * (\y - 1)}
						\draw (\hoek:\nodedist cm) node[vertex,draw,line width=.5pt,fill=black,draw=black] (x-\clr-\y) {};
						\pgfmathtruncatemacro{\yprev}{\y - 1}
						\ifnum\y>1 \draw (x-\clr-\y) -- (x-\clr-\yprev); \fi
					}
				}
				\foreach \x/\clr in {0/0, 1/0, 2/0, 3/1, 4/2, 5/2, 6/3, 7/4, 8/4, 9/4, 10/5, 11/5, 12/6, 13/7, 14/7, 15/8} {
					\draw (v\x) -- (x-\clr-1);
				}
				\draw (v1) -- (v0) to[bend right=20] (v2);
				\draw (v2) -- (v10) -- (v1) -- (v2);
				\draw (v5) -- (v14) (v8) to[bend left=15] (v5);
				\draw (v14) -- (v6) -- (v11) to[bend right=20] (v13);
				\draw (v12) -- (v7);
				\draw (v11) -- (v7) to[bend left=20] (v5);
				\draw (v7) -- (v6) -- (v5) -- (v13) -- (v6);
				\draw[fill=black,draw=black] (1.6,4.5) node[vertex,draw] (x00) {};
				\foreach \x/\clr in {1/0, 3/1, 4.5/2, 6/3, 8/4, 10.5/5, 12/6, 13.5/7, 15/8} {
					\pgfmathsetmacro{\hoek}{90 - 360 * \x / \numverts}
					\begin{scope}[rotate=\hoek]
						\draw (x-\clr-2) ++(0,.25) coordinate (Lctrl-\clr-base);
						\foreach \y in {1,...,4} {
							\pgfmathsetmacro{\nodedist}{1.4 + .25 * (\clr + 4 + \y) + .3 - .12 * \y}
							\coordinate (Lctrl-\clr-\y) at (\nodedist,.3);
							\coordinate (Rctrl-\clr-\y) at (\nodedist,-.8);
						}
					\end{scope}
				}
				\draw[rounded corners] (x-0-2) -- (Lctrl-0-base) -- (x00);
				\draw[rounded corners] (x-1-2) -- (Lctrl-1-base) -- (Lctrl-1-1) -- (x00);
				\draw[rounded corners] (x-2-2) -- (Lctrl-2-base) -- (Rctrl-1-2) -- (Lctrl-1-2) -- (x00);
				\draw[rounded corners] (x-3-2) -- (Lctrl-3-base) -- (Rctrl-2-2) -- (Rctrl-1-3) -- (Lctrl-1-3) -- (x00);
				\draw[rounded corners] (x-4-2) -- (Lctrl-4-base) -- (Rctrl-3-2) -- (Rctrl-2-3) -- (Rctrl-1-4) -- (Lctrl-1-4) -- (x00);
				\draw[rounded corners] (x-5-2) -- (Lctrl-5-base) -- (Lctrl-5-1) -- (Lctrl-6-2) -- (Lctrl-7-3) -- (Lctrl-8-4) -- (Rctrl-8-4) -- (x00);
				\draw[rounded corners] (x-6-2) -- (Lctrl-6-base) -- (Lctrl-6-1) -- (Lctrl-7-2) -- (Lctrl-8-3) -- (Rctrl-8-3) -- (x00);
				\draw[rounded corners] (x-7-2) -- (Lctrl-7-base) -- (Lctrl-7-1) -- (Lctrl-8-2) -- (Rctrl-8-2) -- (x00);
				\draw[rounded corners] (x-8-2) -- (Lctrl-8-base) -- (Lctrl-8-1) -- (Rctrl-8-1) -- (x00);
				\node[anchor=base] at (-1.9,-3.2) {$\decolor(\X)$};
			\end{scope}
		\end{tikzpicture}
		\caption{A worked example illustrating the steps of the decoloring procedure of \autoref{def:decolor}. The labels next to the vertices refer to the color of that vertex.}
		\label{fig:decolor}
	\end{figure}
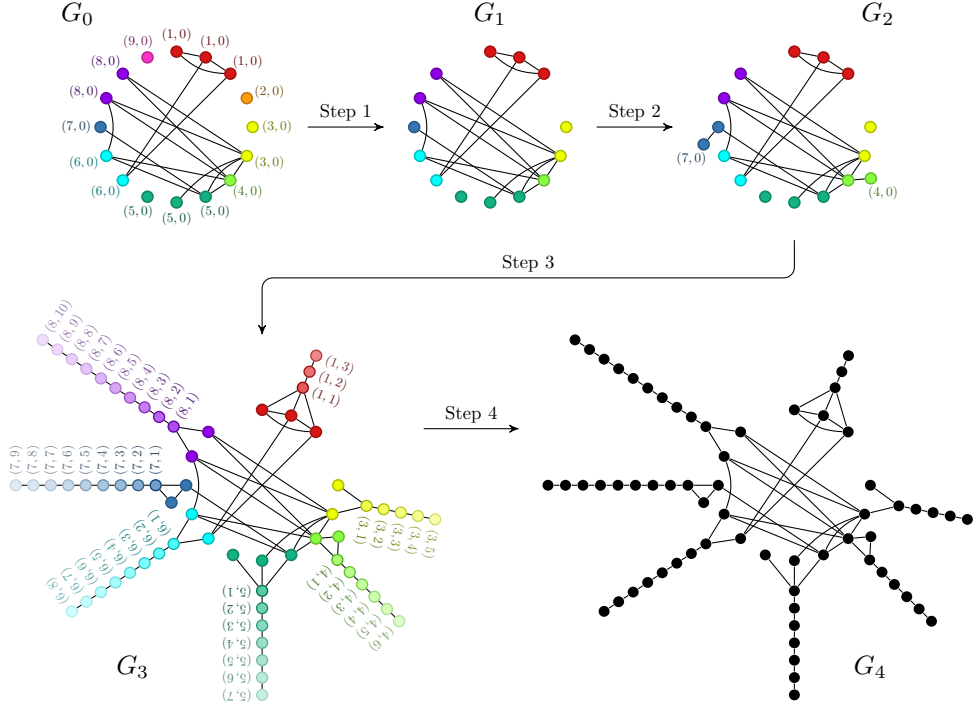
\end{definition}

\begin{theorem}
	\label{thm:decolor}
	The decoloring procedure
	\[ \decolor : \{\text{vertex-colored graphs}\} \to \{\text{uncolored, connected graphs}\} \]
	from \autoref{def:decolor} preserves $\QIso$ and $\Delta$, in the sense that for all vertex-colored graphs $\X,\X'$ there is a $C^*$\nobreakdash-algebra isomorphism
	\[ \decolor : \begin{tikzcd} \!\! \QIso(\X,\X') \arrow[r, "\sim" {yshift=-1pt}] & \QIso(\decolor(\X),\decolor(\X')) \end{tikzcd} \]
	and for all vertex-colored graphs $\X,\X',\X''$ the following diagram commutes:
	\[ \begin{tikzcd}
		\QIso(\X,\X'') \arrow[r, "\Delta_{\X'}"] \arrow[d, "\decolor", "\sim"' {rotate=90, anchor=south, yshift=-1pt, xshift=1pt}]  &  \QIso(\X,\X') \otimes \QIso(\X',\X'') \arrow[d, "\decolor", "\sim"' {rotate=90, anchor=south, yshift=-1pt, xshift=1pt}] \\
		\QIso(\decolor(\X),\decolor(\X'')) \arrow[r, "\Delta_{\decolor(\X')}"]   &  \QIso(\decolor(\X), \decolor(\X')) \otimes \QIso(\decolor(\X'), \decolor(\X'')).
	\end{tikzcd} \]
\end{theorem}

Before proving \autoref{thm:decolor}, we note the following consequence:

\begin{corollary}
	\label{cor:decolor}
	Let $\decolor : \{\text{vertex-colored graphs}\} \to \{\text{graphs}\}$ be the decoloring procedure from \autoref{def:decolor}.
	Then for all vertex-colored graphs $\X,\X'$, we have:
	\begin{enumerate}[label=(\alph*)]
		\item\label{itm:decol:iso} $\X \cong \X'$ if and only if $\decolor(\X) \cong \decolor(\X')$;
		\item\label{itm:decol:qiso} $\X \cong_q \X'$ if and only if $\decolor(\X) \cong_q \decolor(\X')$;
		\item\label{itm:decol:Aut} $\Aut(\X) \cong \Aut(\decolor(\X))$;
		\item\label{itm:decol:Qut} $\Qut(\X) \cong \Qut(\decolor(\X))$.
	\end{enumerate}
\end{corollary}
\begin{proof}
	We prove the statements in a slightly different order (\ref{itm:decol:Qut} before \ref{itm:decol:Aut}).
	\begin{enumerate}[label=(\alph*)]
		\item We have $\X \cong \X'$ if and only if $\QIso(\X,\X')$ has a non-trivial one-dimensional $*$\nobreakdash-representation (see \autoref{rmk:iso-qiso-algebra}), so the result follows because $\QIso(\X,\X') \cong \QIso(\decolor(\X),\decolor(\X'))$.
		
		\item We have $\X \cong_q \X'$ if and only if $\QIso(\X,\X') \neq 0$, so the result follows since $\QIso(\X,\X') \cong \QIso(\decolor(\X),\decolor(\X'))$.
		
		\stepcounter{enumi}
		\item The isomorphism $\decolor : \QIso(\X,\X) \stackrel{\sim}{\longrightarrow} \QIso(\decolor(\X),\decolor(\X))$ shows that $C(\Qut(\X)) \cong C(\Qut(\decolor(\X)))$ as $C^*$\nobreakdash-algebras, and the commutative diagram in \autoref{thm:decolor} shows that this isomorphism preserves the comultiplication, so it is an isomorphism of compact quantum groups.
		
		\addtocounter{enumi}{-2}
		\item This follows from \ref{itm:decol:Qut}, since $\Aut(\X)$ can be recovered from $\Qut(\X)$ (see \autoref{lem:ab}).
		\qedhere
	\end{enumerate}
\end{proof}

\subsection{Color refinement}

The key to our decoloring procedure is a combinatorial technique called \emph{color refinement} (also known as the \emph{$1$-dimensional Weisfeiler--Leman algorithm}), which we will now describe.

Given a vertex-colored graph $\X$ with coloring function $c : V(\X) \to C$, we denote the color classes by $\mathcal C_r := \{i \in V(\X) \mid c(i) = r\}$ for all $r \in C$.
The non-empty color classes form a partition $\mathcal P_c := \{\mathcal C_r \mid r \in C, \mathcal C_r \neq \varnothing\}$ of $V(\X)$. 
We say that two coloring functions $c : V(\X) \to C$, $c' : V(\X) \to C'$ are \emph{equivalent}, written $c \equiv c'$, if $\mathcal P_c = \mathcal P_{c'}$.

\begin{definition}
	\label{def:refinement}
	Let $c : V(\X) \to C$ be a coloring function.
	The \emph{refinement of $c$} is the coloring function $\refine{c} : V(\X) \to C \times \N^C$ given by $\refine{c}(i) = (c(i) , (n_r(i))_{r \in C})$, where $n_r(i) := |N(i) \cap \mathcal C_r|$ denotes the number of neighbors of $i$ with color $r$.
	This refines $c$ in the sense that the corresponding partition $\mathcal P_{\refine{c}}$ is a refinement of $\mathcal P_c$.
	
	The \emph{stable refinement of $c$} is the coloring $\stabref{c}(i) := (\refine[k]{c}(i))_{k=0}^\infty$, where $\refine[k]{c}$ denotes the $k$\nobreakdash-fold iterated refinement of $c$ (with $\refine[0]{c} := c$).
	Since the partition of $V(\X)$ into color classes is refined in every step and there are only finitely many partitions of $V(\X)$, we have $\stabref{c} \equiv \refine[k]{c} \equiv \refine[k + 1]{c}$ for some large enough $k$.
\end{definition}

\begin{remark}
	Alternatively, the stable refinement can be defined as the coarsest stable coloring finer than $c$.
	However, this property only defines $\stabref{c}$ up to equivalence.
	Since quantum isomorphism of colored graphs is not invariant under permutations of the colors, we have to be more precise and define $\stabref{c}$ as a specific coloring function, as we did in \autoref{def:refinement}.%
		\footnote{In the computer science literature, this issue is usually solved in a different way: to compare two graphs $\X$ and $\X'$ using color refinement, just compute the stable refinement of the coloring of the disjoint union $\X \sqcup \X'$.
		This way, passing to an equivalent coloring no longer makes any difference, so the colors can now be represented in a more economical way (for instance, as integers $1,\ldots,k$ instead of infinite sequences).
		Our definition of $\stabref{c}$ as an infinite sequence serves to automatically carry out the same refinement in every graph, so that the colors of $\X$ in the refinement of $\X \sqcup \X'$ do not depend on $\X'$.}
\end{remark}

\begin{definition}
	Let $\X$ and $\X'$ be vertex-colored graphs whose colorings $c : V(\X) \to C$, $c' : V(\X') \to C$ take values in a common set of colors $C$.
	We say that a vertex $i \in V(\X)$ and a vertex $j \in V(\X')$ \emph{are distinguished by color refinement} if and only if $\stabref{c}(i) \neq \stabref{c'}(j)$.
\end{definition}

\begin{remark}
	\label{rmk:degree}
	Vertices of different degrees are always distinguished after one refinement, since we have $\sum_{r \in C} n_r(i) = \deg(i) \neq \deg(j) = \sum_{r \in C} n_r(j)$, so there must be some $r \in C$ such that $n_r(i) \neq n_r(j)$.
\end{remark}

The most important part of the proof of \autoref{thm:decolor} will be to show that the colors of the graph $\X_2$ from \autoref{def:decolor} can be recovered by applying color refinement to the corresponding uncolored graph $\decolor(\X)$.
For this purpose, the first step of our decoloring procedure was to add certain ``gadgets'' to the graph, which can later be used to recognize the original color classes.
This is a common technique in areas related to graph isomorphisms/automorphisms.
We have chosen paths of different lengths as our gadgets, but other gadgets can likely be used as well.

To be able to reason about these paths, we introduce the following (non-standard) terminology.

\begin{definition}
	A \emph{topological path} in a graph $\X$ is a simple path $x_0\cdots x_k$ such that $\deg(x_0),\deg(x_k) \neq 2$ and $\deg(x_1) = \cdots = \deg(x_{k-1}) = 2$.
	In other words, each of the endpoints of the path is either a leaf (a vertex of degree~$1$) or a branch vertex (a vertex of degree~$\geq 3$), and there are no branch vertices on the interior of the path.
\end{definition}

Every vertex of degree $2$ either lies on a unique topological path or on a ``topological loop''.
The following lemma shows that vertices of degree $2$ that lie on a topological path can often be distinguished from one another by looking at the lengths or the endpoints of the (unique) topological paths containing them.

\begin{lemma}
	\label{lem:topological-path}
	Let $\X$ and $\X'$ be graphs, let $p_0\cdots p_k$ and $q_0\cdots q_\ell$ be topological paths in $\X$ and $\X'$, respectively, and let $i \in \{1,\ldots,k-1\}$ and $j \in \{1,\ldots,\ell-1\}$.
	Assume that one of the following is true:
	\begin{enumerate}[label=(\roman*).,ref=(\roman*)]
		\item\label{itm:top:diff-degree} $\{\deg(p_0),\deg(p_k)\} \neq \{\deg(q_0),\deg(q_\ell)\}$, or
		\item\label{itm:top:diff-len} $(\deg(p_0),\deg(p_k)) = (\deg(q_0),\deg(q_\ell))$ and $k \neq \ell$, or
		\item\label{itm:top:diff-pos} $(\deg(p_0),\deg(p_k)) = (\deg(q_0),\deg(q_\ell))$ and $\deg(p_0) \neq \deg(p_k)$ and $i \neq j$.
	\end{enumerate}
	Then $p_i$ and $q_j$ are distinguished by color refinement.
\end{lemma}
\hyphenation{to-po-lo-gi-cal}
\begin{proof}[Proof sketch]
	Proceed by induction.
	The endpoints $p_0$, $p_k$, $q_0$ and $q_\ell$ of the topological paths are distinguished from the interior vertices (and possibly from one another) after one refinement, by a degree argument (see \autoref{rmk:degree}).
	Then, after the second refinement, $p_1$, $p_{k-1}$, $q_1$ and $q_{\ell - 1}$ are distinguished from the remaining interior vertices (and possibly from one another) because they are the only ones having a neighbor in one of the color classes of the endpoints after one refinement.
	Continuing in this way, it is not hard to see that the listed conditions are sufficient for $p_i$ and $q_j$ to be distinguished by color refinement.
	The details are left to the reader.
\end{proof}

The same conclusion cannot necessarily be reached for the endpoints of the topological paths (i.e., $i \in \{0,k\}$, $j \in \{0,\ell\}$), unless we know which other topological paths they connect to.

The following lemma, which forms the core of the proof of \autoref{thm:decolor}, shows that the colors of the graph $\X_2$ after Step~2 of the decoloring procedure can be recovered by applying color refinement to the uncolored graph $\X_3 = \decolor(\X)$.

\begin{lemma}
	\label{lem:decoloring-refinement}
	Let $\X$ and $\X'$ be vertex-colored graphs with colors in $\N_1 \times \{0\}$, and let $c_2 : V(\X_2) \to \N^2$, $c_2' : V(\X_2') \to \N^2$ \textup(resp.{} $c_3 : V(\X_3) \to \{1\}$, $c_3' : V(\X_3') \to \{1\}$\textup) denote the colorings of the respective graphs after Step~2 \textup(resp.{} Step~3\textup) of \autoref{def:decolor}.
	Then for all $i \in V(\X_2) = V(\X_3)$ and all $j \in V(\X_2') = V(\X_3')$, we have
	\[ c_2(i) \neq c_2'(j) \quad \Longrightarrow \quad \stabref{c_3}(i) \neq \stabref{c_3'}(j). \]
\end{lemma}
\begin{proof}
	Denote the vertices that were added to $\X$ and $\X'$ in the decoloring procedure of \autoref{def:decolor} by $x_{(c,r)}$ and $x_{(c,r)}'$, respectively.
	We show that all vertices $i \in V(\X_2)$, $j \in V(\X_2')$ that have different colors are distinguished by color refinement (applied to the trivial colorings of the graphs $\X_3 = \decolor(\X)$ and $\X_3' = \decolor(\X')$).
	For this purpose, let us say that the color $(c,r) \in \N^2$ is \emph{recognized by color refinement} if $\stabref{c_3}(i) \neq \stabref{c_3'}(j)$ whenever exactly one of $i$ and $j$ has color $(c,r)$ in the coloring after Step~2.
	If this is the case, then the color class of $(c,r)$ in the coloring $c_2$ (resp.{} $c_2'$) is a union of color classes of $\stabref{c_3}$ (resp.{} $\stabref{c_3'}$).
	In this setting, we will tacitly identify the colors in the latter union of color classes with $(c,r)$.
	
	We start by categorizing the vertices of degree $1$ and $2$ in $\X_3$ and $\X_3'$.
	First of all, a vertex of degree $1$ in $\X_3$ (resp.{} $\X_3'$) must be of one of the following three types:
	\begin{itemize}
		\item A vertex corresponding to an isolated vertex in $\X$ (resp.{} $\X'$);
		\item The vertex $x_{(c,c+4)}$ (resp.{} $x_{(c',c'+4)}'$) for any color $(c,0) \in C_\X$ (resp.{} $(c',0) \in C_{\X'}$);
		\item The vertex $x_{(0,0)}$ (resp.{} $x_{(0,0)}'$), in case $|C_\X| = 1$ (resp.{} $|C_{\X'}| = 1$).
	\end{itemize}
	Likewise, a vertex of degree $2$ in $\X_3$ (resp.{} $\X_3'$) must be of one of the following four types:
	\begin{itemize}
		\item A vertex corresponding to a vertex of degree $1$ in $\X$ (resp.{} $\X'$);
		\item The vertex $x_{(c,1)}$ (resp.{} $x_{(c',1)}'$) for some color $(c,0) \in C_\X$ (resp.{} $(c',0) \in C_{\X'}$) that occurs exactly once in $\X$ (resp.{} $\X'$);
		\item The vertices $x_{(c,3)},\ldots,x_{(c,c+3)}$ (resp.{} $x_{(c',3)}',\ldots,x_{(c',c'+3)}'$) for any color $(c,0) \in C_\X$ (resp.{} $(c',0) \in C_{\X'}$);
		\item The vertex $x_{(0,0)}$ (resp.{} $x_{(0,0)}'$), in case $|C_\X| = 2$ (resp.{} $|C_{\X'}| = 2$).
	\end{itemize}
	Consequently, a vertex of degree $2$ that lies on a topological path between a branch vertex and a leaf must be one of the following two types:
	\begin{itemize}
		\item The vertices $x_{(c,3)},\ldots,x_{(c,c+3)}$ (resp.{} $x_{(c',3)}',\ldots,x_{(c',c'+3)}'$) for any color $(c,0) \in C_\X$ (resp.{} $(c',0) \in C_{\X'}$);
		\item The vertex $x_{(c,1)}$ (resp.{} $x_{(c',1)}$), in case the color class of $(c,0) \in C_\X$ (resp.{} $(c',0) \in C_{\X'}$) consists of a single, isolated vertex.
	\end{itemize}
	By comparing the parameters of these topological paths, it follows from \autoref{rmk:degree} and \autoref{lem:topological-path} that for all $c \geq 1$, the colors $(c,3),\ldots,(c,c+3)$ are recognized by color refinement.
	It follows that $x_{(c,c+4)}$ (resp.{} $x_{(c,c+4)}'$) is the only vertex of degree $1$ with a neighbor of color $(c,c+3)$, so the color $(c,c+4)$ is also recognized by color refinement.
	Likewise, $x_{(c,2)}$ (resp.{} $x_{(c,2)}'$) is the only branch vertex with a neighbor of color $(c,3)$, so the color $(c,2)$ is also recognized by color refinement.
	To see that the color $(0,0)$ is recognized by color refinement, we distinguish two cases.
	\begin{itemize}
		\item If $|C_\X| = 1$, then $x_{(0,0)}$ is the only vertex of degree $1$ with a neighbor of color $(c,2)$, so $x_{(0,0)}$ is distinguished from all vertices $j \in V(\X') \setminus \{x_{(0,0)}'\}$.
		\item If $|C_\X| > 1$, then $x_{(0,0)}$ is the only vertex which has at least one neighbor in each of the colors $\{(c,2) \mid c \in C_\X\}$, so $x_{(0,0)}$ is distinguished from all vertices $j \in V(\X') \setminus \{x_{(0,0)}'\}$.
	\end{itemize}
	This shows that $\stabref{c_3}(i) \neq \stabref{c_3'}(j)$ whenever $c_2(i) = (0,0)$ and $c_2'(j) \neq (0,0)$.
	Proving the same with the roles of $\X$ and $\X'$ reversed shows that the color $(0,0)$ is recognized by color refinement.
	Since $x_{(c,1)}$ (resp.{} $x_{(c,1)}'$) is the only vertex with color different from $(c,3)$ and $(0,0)$ that has a neighbor of color $(c,2)$, it follows that the color $(c,1)$ is also recognized by color refinement.
	Finally, for every $c \geq 1$, the vertices with color $(c,0)$ are the only ones with color different from $(c,2)$ with a neighbor of color $(c,1)$, so the original colors of $\X$ and $\X'$ are also recognized by color refinement.
\end{proof}

\subsection{Modifications that leave the quantum isomorphism \texorpdfstring{$C^*$\nobreakdash-algebra}{C*-algebra} invariant}
\label{subsec:modifications}

To prove \autoref{thm:decolor}, we show that the quantum isomorphism $C^*$\nobreakdash-algebras are preserved under various modifications to the graph: adding a vertex in a new color (\myautoref{lem:modifications}{itm:mod:isolated}), adding all possible edges between different color classes, assuming that no such edges existed before (\myautoref{lem:modifications}{itm:mod:ST}), and applying color refinement (\autoref{cor:qiso-refinement}).
For this we use the following notation.

\begin{definition}
	Given a graph $\X$, a vertex $i \in V(\X)$ and a vertex set $S \subseteq V(\X)$, we write $e(i,S)$ for the number of edges between $i$ and $S$.
	If $i \in S$, then we also write $\deg_S(i) := e(i,S)$ (the \emph{degree of $i$ relative to $S$}).
\end{definition}

To prove that $\QIso$ is invariant under color refinement, we need the following lemmas.

\begin{lemma}
	\label{lem:local-degree}
	Let $\X$ and $\X'$ be vertex-colored graphs whose colorings $c : V(\X) \to C$, $c' : V(\X') \to C$ take values in a common set of colors $C$, let $u = (u_{ij})_{i \in V(\X), j \in V(\X')}$ be the generating matrix of $\QIso(\X,\X'$), and let $S \subseteq V(\X)$, $S' \subseteq V(\X')$ be vertex sets such that $u_{k\ell} = 0$ whenever $k \in S$, $\ell \notin S'$ or $k \notin S$, $\ell \in S'$.
	Then for all $i \in S$, $j \in S'$ we have $\deg_S(i) \, u_{ij} = \deg_{S'}(j) \, u_{ij}$.
\end{lemma}
\begin{proof}
	We have
	\begin{align*}
		\deg_S(i) \, u_{ij} &= \sum_{k \in N(i) \cap S} \sum_{\ell \in V(\X)} u_{k\ell} u_{ij} \\
		&= \sum_{k \in N(i) \cap S} \sum_{\ell \in N(j) \cap S'} u_{k\ell} u_{ij} \\
		&= \sum_{k \in V(\X)} \sum_{\ell \in N(j) \cap S'} u_{k\ell} u_{ij} \\
		&= \deg_{S'}(j) \, u_{ij},
	\end{align*}
	where the first and fourth equality follow since $\sum_{\ell \in V(\X)} u_{k\ell} = 1$ for all $k$ and $\sum_{k \in V(\X)} u_{k\ell} = 1$ for all $\ell$, respectively, and the second and third equality follow because we have $u_{k\ell} = 0$ whenever $k \in S$, $\ell \notin S'$ or $k \notin S$, $\ell \in S'$ (by assumption), as well as $u_{k\ell} u_{ij} = 0$ whenever $ik \in E(\X)$, $j\ell \notin E(\X)$ or $ik \notin E(\X)$, $j\ell \in E(\X)$ (by definition of $\Qut(\X)$).
\end{proof}

\begin{lemma}
	\label{lem:qiso-refinement}
	Let $\X$ and $\X'$ be vertex-colored graphs whose colorings $c : V(\X) \to C$, $c' : V(\X') \to C$ take values in a common set of colors $C$.
	Then $\QIso((\X,c),(\X',c')) = \QIso((\X,\refine{c}),(\X',\refine{c'}))$.
\end{lemma}
\begin{proof}
	Let $u = (u_{ij})_{i,j}$ be the generating matrix of $\QIso((\X,c),(\X',c'))$.
	We show that $u_{ij} = 0$ whenever $\refine{c}(i) \neq \refine{c'}(j)$.
	If $c(i) \neq c'(j)$, then this follows immediately from the definition of $\QIso(\X,\X')$, so assume that $c(i) = c'(j)$.
	Then there is some color $r$ such that $i$ and $j$ have a different number of neighbors with color $r$.
	Define $S,T \subseteq V(\X)$ and $S',T' \subseteq V(\X')$ by
	\begin{align*}
		& S := \{k \in V(\X) \mid c(k) = c(i)\} && \text{and} && T := \{k \in V(\X) \mid c(k) = r\}; \\
		& S' := \{\ell \in V(\X') \mid c'(\ell) = c(j)\} && \text{and} && T' := \{\ell \in V(\X) \mid c(k) = r\}.
	\end{align*}
	We distinguish two cases.
	\begin{itemize}
		\item If $r= c(i) = c'(j)$, then $i$ and $j$ have a different number of neighbors in their own color classes $S = T$ and $S' = T'$.
		In other words, we have $\deg_S(i) \neq \deg_{S'}(j)$, so it follows from \autoref{lem:local-degree} that $u_{ij} = 0$.
		
		\item If $r \neq c(i) = c'(j)$, then $i$ and $j$ have a different number of neighbors in the color classes $T \niton i$ and $T' \niton j$, respectively (both corresponding to the color $r$).
		Equivalently, we have $e(i,T) \neq e(j,T')$.
		Note that $e(i,T) = \deg_{S \cup T}(i) - \deg_S(i)$ and $e(j,T') = \deg_{S' \cup T'}(j) - \deg_S(j)$, so it follows from \autoref{lem:local-degree} that
		\begin{align*}
			e(i,T) \, u_{ij} \, &= \, \deg_{S \cup T}(i) \, u_{ij} \, - \, \deg_S(i) \, u_{ij} \\
			&= \, \deg_{S' \cup T'}(j) \, u_{ij} \, - \, \deg_{S'}(j) \, u_{ij} \, = \, e(j,T') \, u_{ij}.
		\end{align*}
		Since $e(i,T) \neq e(j,T')$, it follows that $u_{ij} = 0$.
	\end{itemize}
	Either way, we find that $u_{ij} = 0$, as claimed.
	Now it is easy to see that $\QIso((\X,c),(\X',c')) = \QIso((\X,\refine{c}),(\X',\refine{c'}))$.
\end{proof}

\begin{corollary}
	\label{cor:qiso-refinement}
	Let $\X$ and $\X'$ be vertex-colored graphs whose colorings $c : V(\X) \to C$, $c' : V(\X') \to C$ take values in a common set of colors $C$.
	Then $\QIso((\X,c),(\X',c')) = \QIso((\X,\stabref{c}),(\X',\stabref{c'}))$.
\end{corollary}
\begin{proof}
	Choose $k \in \N$ sufficiently large so that $\refine[k]{c} \equiv \stabref{c}$ and $\refine[k]{c'} \equiv \stabref{c'}$.
	Then we have $\stabref{c}(i) = \stabref{c}(j)$ if and only if $\refine[k]{c}(i) = \refine[k]{c}(j)$, so the generating matrices of the quantum isomorphism $C^*$\nobreakdash-algebras $\QIso((\X,\stabref{c}),(\X',\stabref{c'}))$ and $\QIso((\X,\refine[k]{c}),(\X',\refine[k]{c'}))$ satisfy the exact same relations.
	Therefore we have $\QIso((\X,\stabref{c}),(\X',\stabref{c'})) = \QIso((\X,\refine[k]{c}),(\X',\refine[k]{c'}))$.
	By \autoref{lem:qiso-refinement}, the result follows.
\end{proof}

\hyphenation{generali-zation}
To add vertices and edges, we use the following lemma, which is a generalization of \cite[Lemma 3.1]{DKRSZ} (from $\Qut$ to $\QIso$ algebras).

\begin{lemma}[{cf.{} \cite[Lemma 3.1]{DKRSZ}}]
	\label{lem:modifications}
	Let $\X$ and $\X'$ be vertex-colored graphs whose colorings $c : V(\X) \to C$, $c' : V(\X') \to C$ take values in a common set of colors $C$.
	\begin{enumerate}[label=\textup(\Roman*\textup)]
		\item\label{itm:mod:isolated}
		Let $\tilde C := C \sqcup \{r\}$ for some new color $r \notin C$, and let $\tilde \X$ \textup(resp.{} $\tilde \X'$\textup) be the vertex-colored graph obtained from $\X$ \textup(resp.{} $\X'$\textup) by adding an isolated vertex with color $r$.
		Then $\QIso(\X,\X') \cong \QIso(\tilde \X,\tilde \X')$.
		Furthermore, this operation preserves the co-composition \textup(in the sense of \autoref{thm:decolor}\textup).
		
		\item\label{itm:mod:ST}
		Let $r_1,r_2 \in C$ be distinct colors such that both $\X$ and $\X'$ have no edges with one endpoint colored $r_1$ and the other endpoint colored $r_2$, and let $\tilde \X$ \textup(resp.{} $\tilde \X'$\textup) be the vertex-colored graph obtained from $\X$ \textup(resp.{} $\X'$\textup) by adding all possible such edges.
		Then $\QIso(\X,\X') = \QIso(\tilde \X,\tilde \X')$.
		\textup(In particular, this operation preserves the co-composition.\textup)
	\end{enumerate}
\end{lemma}

The proof of \autoref{lem:modifications} is exactly the same as the proof of \cite[Lemma 3.1]{DKRSZ} and is therefore omitted.

\subsection{Proof of the decoloring theorem}

We now come to the proof of the main result of this section.

\begin{proof}[{Proof of \autoref{thm:decolor}}]
	We prove that each of the steps from \autoref{def:decolor} preserves $\QIso$ and $\Delta$ (in the sense described in \autoref{thm:decolor}).
	For $s \in \{1,2,3\}$, let $\X_s$ denote the graph after the $s$-th step of the decoloring procedure applied to $\X$.
	
	For Step 1, let $\X,\X',\X''$ be vertex-colored graphs with colors in the set $\N_1 \times \{0\}$.
	To prove that Step 1 preserves $\QIso$, we distinguish two cases.
	\begin{itemize}
		\item If $\X$ and $\X'$ use the same colors (i.e.{} $C_\X = C_{\X'}$, in the notation of \autoref{def:decolor}), then Step 1 adds the same vertices to $\X$ and $\X'$ (in the same colors), so it follows from \myautoref{lem:modifications}{itm:mod:isolated} that $\QIso(\X,\X') \cong \QIso(\X_1,\X_1')$.
		\item If one of $\X$ and $\X'$ uses a color that does not occur in the other, then the same is true for $\X_1$ and $\X_1'$, and so it follows from \autoref{rmk:qiso-vertex-counts} that $\QIso(\X,\X') = \QIso(\X_1,\X_1') = 0$.
	\end{itemize}
	Likewise, to prove that Step 1 preserves $\Delta$, we distinguish two cases.
	\begin{itemize}
		\item If $\X$, $\X'$ and $\X''$ all use the same colors, then Step 1 adds the same vertices to each of $\X$, $\X'$ and $\X''$ (in the same colors), so it follows from \myautoref{lem:modifications}{itm:mod:isolated} that Step 1 preserves $\Delta$.
		\item If the sets of colors used by $\X$, $\X'$ and $\X''$ are not all equal, then at least one of $\QIso(\X,\X')$ and $\QIso(\X',\X'')$ is trivial, so we have $\QIso(\X,\X') \otimes \QIso(\X',\X'') = 0$, and therefore also $\QIso(\X_1,\X_1') \otimes \QIso(\X_1',\X_1'') = 0$.
		It follows that Step 1 preserves $\Delta$.
	\end{itemize}
	
	It follows immediately from \myautoref{lem:modifications}{itm:mod:ST} that Step 2 preserves $\QIso$ and $\Delta$, since this step adds all possible edges between fixed pairs of colors (the same pairs for all graphs).
	
	To show that Step 3 preserves $\QIso$ and $\Delta$, it suffices to show that $\QIso(\X_2,\X_2') = \QIso(\X_3,\X_3')$, since equality of the generating matrices implies that the co-composition is preserved.
	Let $c_2 : V(\X_2) \to \N^2$ and $c_3 : V(\X_3) \to \{1\}$ denote the colorings of $\X_2$ and $\X_3$ (and likewise $c_2'$, $c_3'$ for $\X_2'$ and $\X_3'$).
	By \autoref{lem:decoloring-refinement} and universality of the respective quantum isomorphism $C^*$\nobreakdash-algebras, we have natural maps
	\[ \QIso(\X_3,\X_3') \to \QIso(\X_2,\X_2') \to \QIso((\X_3,\stabref{c_3}),(\X_3',\stabref{c_3'})), \]
	where each map preserves the generating matrix $(u_{ij})_{i \in V(\X_2), j \in V(\X_2')}$ but adds more constraints.
	But by \autoref{cor:qiso-refinement}, we have $\QIso(\X_3,\X_3') = \QIso((\X_3,\stabref{c_3}),(\X_3',\stabref{c_3'}))$, so in fact we have equality throughout:
	\[ \QIso(\X_3,\X_3') = \QIso(\X_2,\X_2') = \QIso((\X_3,\stabref{c_3}),(\X_3',\stabref{c_3'})). \qedhere \]
\end{proof}

\section{An uncolored graph whose quantum automorphism group is the dual of a solution group}
\label{sec:uncolored}

In this short section, we use the results from the previous sections to define, for every linear system $Mx = b$ over $\Z_2$, an uncolored graph $G'(M,b)$ such that $\Qut(G'(M,b))$ is isomorphic to the dual of the solution group $\Gamma(M,0)$.

\begin{definition}
	\label{defgraph'}
	For $M \in \Z_2^{m \times n}$ and $b \in \Z_2^m$, we define $G'(M,b)$ to be the uncolored graph obtained by taking the graph $G(M,b)$ from \autoref{defgraph} and applying the decoloring procedure of \autoref{def:decolor}.
\end{definition}

\begin{theorem}
	\label{thm:allsolgroupsuncolored}
	Let $M\in \Z_2^{m\times n}$ and $b \in \Z_2^m$.
	Then $\Qut(G'(M,b))$ is isomorphic to the dual of the solution group $\Gamma(M,0)$.
\end{theorem}
\begin{proof}
	By \autoref{cor:vertexcolored}, $\Qut(G(M,b))$ is isomorphic to the dual of $\Gamma(M,0)$.
	Since $G'(M,b) = \decolor(G(M,b))$, it follows from \myautoref{cor:decolor}{itm:decol:Qut} that
	\[ \Qut(G'(M,b)) \cong \Qut(G(M,b)) \cong (\Gamma(M,0))\: \widehat{}. \qedhere \]
\end{proof}

\begin{remark}
	Although we will have no use for this, we point out that it is also possible to construct graphs whose quantum automorphism group is isomorphic to the dual of any \emph{inhomogeneous} solution group.
	To that end, we note that every inhomogeneous solution group is isomorphic to a homomgeneous solution group:
\end{remark}

\begin{lemma}
	\label{lem:inhom2hom}
	Let $M \in \mathbb{Z}_2^{m \times n}$ and $b \in \mathbb{Z}_2^m$ such that $b \ne 0$.
	Then there exists $M' \in \mathbb{Z}_2^{(m+1) \times (n+2)}$ such that $\Gamma(M,b) \cong \Gamma(M',0)$.
\end{lemma}
\begin{proof}
	We form $M'$ by adding two columns and then one row to $M$ as follows:
	\[M' = \begin{pmatrix} M & \one + b & \one \\ \boldsymbol{0} & 1 & 1\end{pmatrix},\]
	where $\one$ denotes the all ones vector. Considering now the solution group $\Gamma(M',0)$ with generators $\{x_i \ | \ i \in [n+2]\}$, we see that $x_{n+2}$ commutes with all generators (this will take the place of $\gamma$ in $\Gamma(M,b)$) and that $x_{n+1} = x_{n+2}$. Therefore, for $k \in [m]$, the relation $\prod_{i \in S_k(M')} x_i = 1$ becomes
	\[\prod_{i \in S_k(M)} x_i = \begin{cases}1 & \text{if } b_k = 0\\ x_{n+2} & \text{if } b_k = 1.\end{cases}\]
	It is now easy to see that if $\{x'_i \ | \ i \in [n]\} \cup \{\gamma\}$ are the generators of $\Gamma(M,b)$, then the map $x_i \mapsto x'_i$ for $i \in [n]$ and $x_{n+2} \mapsto \gamma$ extends to an isomorphism of $\Gamma(M,b)$ and $\Gamma(M',0)$.
\end{proof}

\begin{corollary}
	The dual quantum group of every homogeneous or inhomogeneous solution group can be obtained as the quantum automorphism group of some \textup(finite, uncolored\textup) graph.
\end{corollary}
\begin{proof}
	Immediate from \autoref{thm:allsolgroupsuncolored} and \autoref{lem:inhom2hom}.
\end{proof}

\section{A sequence of perfect homogeneous solution groups}
\label{sec:perfect}

In the previous section, we showed how to construct a graph whose quantum automorphism group is isomorphic to the dual of the homogeneous solution group of any given linear constraint system.
To get a graph with quantum symmetry and trivial automorphism group, we must find a linear constraint system whose homogeneous solution group is perfect (i.e.{} has trivial abelianization).
For this we use the following construction.

\begin{definition}\label{deflinsysgroup}
	Let $H$ be a finite group.
	Let $O_2(H) = \{h_1,\ldots,h_\ell\} \subseteq H$ denote the set of elements of order $2$ in $H$, and let
	\begin{align*}
		T_2(H) := \big\{\{h_i,h_j,h_k\} \, \mid \, &h_i,h_j,h_k \in O_2(H), \ [h_i,h_j] = [h_i,h_k] = [h_j,h_k] = 1, \\&\ \text{and} \ h_ih_jh_k = 1\big\}
	\end{align*}
	be the set of all triples\footnote{Note that every set $\{h_i,h_j,h_k\}$ in $T_2(H)$ consists of three different elements. Indeed, if (say) $h_i = h_j$, then one has $h_ih_j = h_i^2 = 1$, and therefore $h_k = h_ih_jh_k = 1$, contrary to the assumption that $h_k$ has order $2$.} of elements of order $2$ that commute pairwise and multiply to the identity.
	Then we define $M_H \in \Z_2^{T_2(H) \times O_2(H)}$ to be the coefficient matrix of the linear system
	\[ x_{h_i} + x_{h_j} + x_{h_k} \equiv 0 \pmod{2} \qquad \text{for all $\{h_i,h_j,h_k\} \in T_2(H)$}. \]
	In other words, the variables $x_{h_i}$ in this linear system correspond to elements $h_i \in H$ of order $2$, and the equations state that $x_{h_i} + x_{h_j} + x_{h_k} \equiv 0 \pmod{2}$ whenever $h_i$, $h_j$ and $h_k$ commute and multiply to the identity.
\end{definition}

The main result of this section is that $\Gamma(M_{A_n},0)$ is perfect for all $n \geq 7$, where $A_n$ denotes the alternating group.

\begin{proposition}
	\label{prop:MH-group}
	For every finite group $H$, the homogeneous solution group $\Gamma(M_H,0)$ is the finitely presented group generated by elements $x_{h_i}$ corresponding to all elements $h_1,\ldots,h_\ell \in H$ of order $2$ in $H$, with the following relations:
	\begin{align}
		x_{h_i}^2 &= 1 \qquad \text{for all $i \in [\ell]$}; \label{eqn:MH-order-2} \\
		[x_{h_i},x_{h_j}] &= 1 \qquad \text{whenever $[h_i,h_j] = 1$}; \label{eqn:MH-commutes} \\
		x_{h_i}x_{h_j}x_{h_k} &= 1 \qquad \text{whenever $h_i$, $h_j$ and $h_k$ commute pairwise} \label{eqn:MH-triples}\\
  &\phantom{= 1} \qquad \, \text{and multiply to $1$}. \nonumber
	\end{align}
\end{proposition}
\begin{proof}
	The only thing that is not immediate from the definitions is that \relautoref{solgroup:rel2} from \autoref{def:solgroup} simplifies to \eqref{eqn:MH-commutes}.
	To see this, note that for every pair $h_i,h_j \in O_2(H)$ with $h_i \neq h_j$ and $[h_i,h_j] = 1$, the product $h_k := h_ih_j$ ($ = h_jh_i$) also has order $2$ and commutes with $h_i$ and $h_j$, and we have $h_ih_jh_k = 1$, hence $\{h_i,h_j,h_k\} \in T_2(H)$.
	In other words, in the linear system given by $M_H$, two variables $x_{h_i}$ and $x_{h_j}$ occur in the same equation together if and only if $[h_i,h_j] = 1$.
\end{proof}
Since $h_1,\ldots,h_\ell \in H$ satisfy the relations of \autoref{prop:MH-group}, the following corollaries are immediate.\leavevmode
\begin{corollary}
	\label{cor:natural-hom}
	For every finite group $H$, there is a homomorphism $\phi : \Gamma(M_H,0) \to H$ such that $\phi(x_{h_i}) = h_i$ for all $h_i \in O_2(H)$.
\end{corollary}
\begin{corollary}
	\label{cor:non-trivial-non-commutative}
	Let $H$ be a finite group.
	\begin{enumerate}[label=(\alph*)]
		\item If $H$ has an element of order $2$, then $\Gamma(M_H,0)$ is non-trivial.
		\item If $H$ has two non-commuting elements of order $2$, then $\Gamma(M_H,0)$ is non-commutative.
	\end{enumerate}
\end{corollary}

In particular, for the alternating group, it follows that $\Gamma(M_{A_n},0)$ is non-trivial for all $n \geq 4$ and non-commutative for all $n \geq 5$.

To prove that $\Gamma(M_{A_n},0)$ is perfect for all $n \geq 7$, we first consider the case $n = 7$.
Here \eqref{eqn:MH-triples} gives us the following relations\footnote{One can show that all triples in $T_2(A_7)$ are of the form \eqref{A7:rel1} or \eqref{A7:rel2}, but we won't need that in the proof.}:
\begin{align}
	x_{(ab)(cd)}x_{(ab)(ef)}x_{(cd)(ef)} &= 1 \, \text{for all disjoint $\{a,b\}$, $\{c,d\}$, $\{e,f\} \in \textstyle\binom{[7]}{2}$}; \label{A7:rel1}\\
	x_{(ab)(cd)}x_{(ac)(bd)}x_{(ad)(bc)} &= 1 \, \text{for all $\{a,b,c,d\} \in \textstyle\binom{[7]}{4}$}. \label{A7:rel2}
\end{align}
Using these, we will show that $\Gamma(M_{A_7},0)$ is perfect.
Arguably the easiest way to do so is to write down the matrix $M_{A_7}$ and compute its rank.
Using a computer algebra system, one can check that $M_{A_7}$ is a $140 \times 105$ matrix over $\Z_2$ which has rank $105$, so it follows that the linear system $M_{A_7}x = 0$ has no non-trivial solutions.
From this it can easily be deduced that $\Gamma(M_{A_7},0)$ is perfect.
However, this proof is not very illuminating, since it hides the subtleties of the finitely presented group $\Gamma(M_{A_7},0)$ in a computation that is easy for a computer but too large for most humans to fathom.
Instead, we give a human-readable proof of a combinatorial flavour.
For this we use the Kneser graphs $K(S,2)$ over an arbitrary set $S$, which we define as follows.

\begin{definition}
	Let $S$ be a finite set and let $\ell \in \N$.
	The \emph{\textup(labeled\textup) Kneser graph} $K(S,\ell)$ is the graph with vertex set $V(K(S,\ell)) := \binom{S}{\ell}$ (the set of all $\ell$-element subsets of $S$), where two vertices are adjacent if and only if the corresponding subsets are disjoint.
	If $S = [n]$, we simply write $K(n,\ell)$ for the Kneser graph $K([n],\ell)$.
\end{definition}

\begin{theorem}
	\label{thm:A7}
	The homogeneous solution group $\Gamma(M_{A_7},0)$ is perfect.
\end{theorem}
\begin{proof}
	We show that the abelianization of $\Gamma(M_{A_7},0)$ is trivial. This is equivalent to showing that the system of equations \eqref{A7:rel1} and \eqref{A7:rel2} has no non-trivial solutions for $x_{(ab)(cd)}\in \{\pm 1\}$, since this proves that the group has no non-trivial one-dimensional representations.
	
	Consider the Kneser graph $K(7,2)$.
	The graph has vertices $\{a,b\}$ with $a,b \in [7]$ and $a \neq b$, where two vertices are connected if and only if the corresponding subsets are disjoint.
	Therefore, the edges of $K(7,2)$ are of the form $\{a,b\}\{c,d\}$, where $|\{a,b,c,d\}|=4$.
	
	To each edge $\{a,b\}\{c,d\}$ in $K(7,2)$, associate the variable $x_{(ab)(cd)}\in \{\pm 1\}$.
	We color the edge $\{\{a,b\}, \{c,d\}\}$ red if $x_{(ab)(cd)} = -1$ and green otherwise.
	We will show that if there is a red edge in $K(7,2)$, then it is impossible to color all edges of the graph consistent with Equations \eqref{A7:rel1} and \eqref{A7:rel2}.
	Note that Equation \eqref{A7:rel1} states that every triangle in $K(7,2)$ either has $0$ or $2$ red edges, whereas Equation \eqref{A7:rel2} states something similar for certain other edge configurations in $K(7,2)$.\footnote{A graph-theoretic interpretation of \eqref{A7:rel2} is that, for every triple of edges such that each pair is at distance two (where we define the distance between two edges to be the minimum of the distances of their endpoints), either $0$ or $2$ of them are red. But we won't make use of this characterization.}
	
	Assume for the sake of contradiction that we have a coloring of the edges of $K(7,2)$ that is consistent with \eqref{A7:rel1} and \eqref{A7:rel2} with at least one red edge $xy \in E(K(7,2))$.
	Since $K(7,2)$ is strongly regular with parameters $(21, 10,3,6)$, we know that $xy$ is part of three triangles.
	By \eqref{A7:rel1}, each of these triangles must have exactly two red edges, so one of the following must be true:
	\begin{enumerate}[label=(\alph*)]
		\item\label{proof:perfecta} All three additional red edges (in the three triangles containing $xy$) are incident to the same endpoint of $xy$;
		\item\label{proof:perfectb} Two of the additional red edges (in the three triangles containing $xy$) are incident to one of the endpoints and the remaining red edge is incident to the other.
	\end{enumerate}
	These two cases are illustrated in \autoref{figure1}.
	We will show that both cases are not possible.
	
	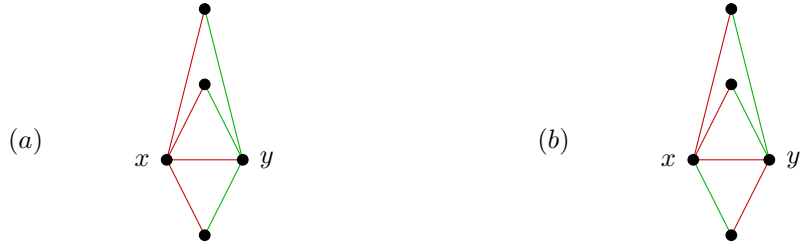
\begin{figure}[H]
		\centering
		\begin{tikzpicture}[scale=.9]
			\draw[Red] (0,0)--(1,0);\draw[Red] (0,0)--(0.5,1);\draw[Red] (0,0)--(0.5,2);\draw[Red] (0,0)--(0.5,-1);
			\draw[Green] (1,0)--(0.5,-1);\draw[Green] (1,0)--(0.5,1);\draw[Green] (1,0)--(0.5,2);
			\draw[black,fill=black] (0,0) circle (2pt);\draw[black,fill=black] (1,0) circle (2pt);\draw[black,fill=black] (0.5,1) circle (2pt);\draw[black,fill=black] (0.5,2) circle (2pt);\draw[black,fill=black] (0.5,-1) circle (2pt);
			\coordinate[label=left:$x$] (1) at (-0.1,0);\coordinate[label=right:$y$] (1) at (1.1,0);
			\coordinate[label=left:$(a)$] (1) at (-1.2,1.75);
		\end{tikzpicture}
		\hspace{22mm}
		\begin{tikzpicture}[scale=.9]
			\draw[Red] (0,0)--(1,0);\draw[Red] (0,0)--(0.5,1);\draw[Red] (0,0)--(0.5,2);\draw[Green] (0,0)--(0.5,-1);
			\draw[Red] (1,0)--(0.5,-1);\draw[Green] (1,0)--(0.5,1);\draw[Green] (1,0)--(0.5,2);
			\draw[black,fill=black] (0,0) circle (2pt);\draw[black,fill=black] (1,0) circle (2pt);\draw[black,fill=black] (0.5,1) circle (2pt);\draw[black,fill=black] (0.5,2) circle (2pt);\draw[black,fill=black] (0.5,-1) circle (2pt);
			\coordinate[label=left:$x$] (1) at (-0.1,0);\coordinate[label=right:$y$] (1) at (1.1,0);
			\coordinate[label=left:$(b)$] (1) at (-1.2,1.75);
		\end{tikzpicture}
		\caption{\hyperref[proof:perfecta]{Case (a)} (left) and \hyperref[proof:perfectb]{Case (b)} (right), up to swapping $x$ and $y$.}
		\label{figure1}
	\end{figure}
	
	Note that the neighborhood of a vertex $v = \{a,b\} \in V(K(7,2))$ is the Kneser graph $K([7] \setminus v,2)$, which is isomorphic to the Petersen graph.
	Consider the vertex-and-edge-coloring of $N(v)$ defined in the following way:
	\begin{itemize}
		\item The edge colors of $N(v)$ are the colors inherited from $K(7,2)$;
		\item Every vertex $w \in N(v)$ is colored according to the color of the edge $vw$ in $K(7,2)$.
	\end{itemize}
	For example, the coloring of $N(y)$ (resp.{} $N(x)$) obtained from the coloring in \myautorefstar{figure1}{(a)} (resp.{} \myautorefstar{figure1}{(b)}) is depicted in \autoref{figure2}. 
	
	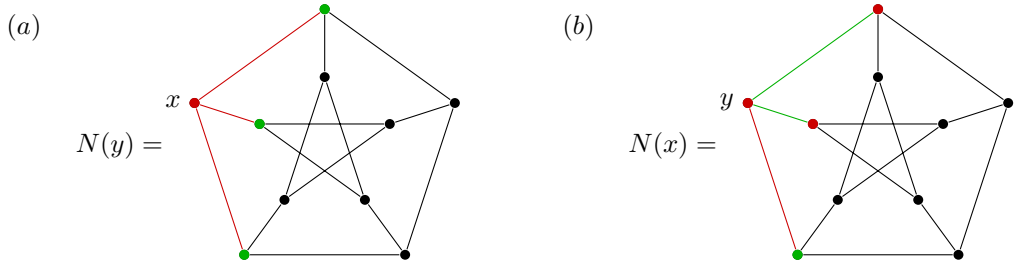
\begin{figure}[H]
		\centering
		\begin{tikzpicture}[scale=0.53]
			\foreach \x in {1,...,5} {
				\pgfmathsetmacro{\curangle}{18 + 72 * (\x - 1)}
				\node[vertex] (I\x) at (\curangle:1.5cm) {};
				\node[vertex] (O\x) at (\curangle:3cm) {};
			}
			\draw (I1) -- (I3) -- (I5) -- (I2) -- (I4) -- (I1);
			\draw (O4) -- (O5) -- (O1) -- (O2);
			\draw[Red] (O2) -- (O3) -- (O4);
			\draw[Red] (I3) -- (O3);
			\draw (I1) -- (O1) (I2) -- (O2) (I4) -- (O4) (I5) -- (O5);
			\node[vertex,emph,Red] at (O3) {};
			\node[vertex,emph,Green] at (O2) {};
			\node[vertex,emph,Green] at (O4) {};
			\node[vertex,emph,Green] at (I3) {};
			\coordinate[label=left:$x$] (1) at (162:3.1cm);
			\coordinate[label=left:$(a)$] (1) at (-5,2.6);
			\node at (-4.5,0) {$N(y) = $};
		\end{tikzpicture}
		\hspace{8mm}
		\begin{tikzpicture}[scale=0.53]
			\foreach \x in {1,...,5} {
				\pgfmathsetmacro{\curangle}{18 + 72 * (\x - 1)}
				\node[vertex] (I\x) at (\curangle:1.5cm) {};
				\node[vertex] (O\x) at (\curangle:3cm) {};
			}
			\draw (I1) -- (I3) -- (I5) -- (I2) -- (I4) -- (I1);
			\draw (O4) -- (O5) -- (O1) -- (O2);
			\draw[Green] (O2) -- (O3) -- (I3);
			\draw[Red] (O3) -- (O4);
			\draw (I1) -- (O1) (I2) -- (O2) (I4) -- (O4) (I5) -- (O5);
			\node[vertex,emph,Red] at (O3) {};
			\node[vertex,emph,Red] at (O2) {};
			\node[vertex,emph,Green] at (O4) {};
			\node[vertex,emph,Red] at (I3) {};
			\coordinate[label=left:$y$] (1) at (162:3.1cm);
			\coordinate[label=left:$(b)$] (1) at (-5,2.6);
			\node at (-4.5,0) {$N(x) = $};
		\end{tikzpicture}
		\caption{Left: the vertex-and-edge coloring of the neighborhood $N(y)$ from \myautorefstar{figure1}{(a)}. Right: the vertex-and-edge coloring of the neighborhood $N(x)$ from \myautorefstar{figure1}{(b)}. Black vertices and edges have unknown color.}
		\label{figure2}
	\end{figure}
	
	\noindent
	Now, we make two useful observations.
	\begin{itemize}
		\item[(i)] In the vertex-and-edge-colored neighborhood of a vertex $v$, an edge and its endpoints together correspond to a triangle in $K(7,2)$.
		Therefore one of the following must be true: the edge and its endpoints are all green
			(~\begin{tikzpicture}[scale=0.5,baseline=-2.3pt]
				\node[vertex,emph,Green] (L) at (0,0) {};
				\node[vertex,emph,Green] (R) at (1,0) {};
				\draw[Green] (L) -- (R);
			\end{tikzpicture}~),
		or the edge is green and its endpoints are red
			(~\begin{tikzpicture}[scale=0.5,baseline=-2.3pt]
				\node[vertex,emph,Red] (L) at (0,0) {};
				\node[vertex,emph,Red] (R) at (1,0) {};
				\draw[Green] (L) -- (R);
			\end{tikzpicture}~),
		or the edge is red and the endpoints have opposite colors
			(~\begin{tikzpicture}[scale=0.5,baseline=-2.3pt]
				\node[vertex,emph,Green] (L) at (0,0) {};
				\node[vertex,emph,Red] (R) at (1,0) {};
				\draw[Red] (L) -- (R);
			\end{tikzpicture}~).
		Note that endpoints of green edges have the same color and endpoints of red edges have different colors.
		It follows that the number of red edges along any walk must be even if the endpoints of the walk have the same color and odd if the endpoints of the walk have different colors.
		In particular, the number of red edges along any cycle must be even.
		
		\item[(ii)] The triples of edges corresponding to Equation \eqref{A7:rel2} partition the edge set of $K([7] \setminus v,2)$ as depicted in \autoref{figure3}.
		By Equation \eqref{A7:rel2}, in each of the classes of this partition, either $0$ or $2$ edges are colored red in the vertex-and-edge-coloring of $N(v)$.
		
		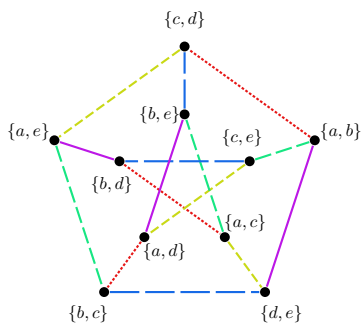
\begin{figure}[H]
			\centering
			\begin{tikzpicture}[scale=0.6,thick]
				\useasboundingbox (-4.3,-2.9) rectangle (4.3,3.4);
				\foreach \x in {18,90,162,234,306} {
					\draw[black,fill=black] (\x:3cm) node[vertex] (O\x) {};
					\draw[black,fill=black] (\x:1.5cm) node[vertex] (I\x) {};
					\pgfmathtruncatemacro{\xa}{mod(2 * (\x - 18) / 72, 5) + 1}
					\pgfmathtruncatemacro{\xb}{mod(\xa, 5) + 1}
					\pgfmathtruncatemacro{\lxab}{min(\xa,\xb)}
					\pgfmathtruncatemacro{\rxab}{max(\xa,\xb)}
					\ifnum\x=90
						\node[scale=.65] at (\x:3.35cm) {$\{\letter\lxab,\letter\rxab\}$};
					\else
						\node[scale=.65] at (\x:3.6cm) {$\{\letter\lxab,\letter\rxab\}$};
					\fi
					\pgfmathtruncatemacro{\xc}{mod(\xa + 3, 5) + 1}
					\pgfmathtruncatemacro{\xd}{mod(\xb, 5) + 1}
					\pgfmathtruncatemacro{\lxcd}{min(\xc,\xd)}
					\pgfmathtruncatemacro{\rxcd}{max(\xc,\xd)}
					\pgfmathsetmacro{\xx}{\x + 20}
					\node[scale=.65] at (\xx:1.6cm) {$\{\letter\lxcd,\letter\rxcd\}$};
				}
				\definecolor{myc1}{hsb}{0,.9,.9}
				\definecolor{myc2}{hsb}{.18,.9,.85}
				\definecolor{myc3}{hsb}{.42,.9,.9}
				\definecolor{myc4}{hsb}{.6,.9,.9}
				\definecolor{myc5}{hsb}{.8,.9,.9}
				\draw[myc1,dash pattern=on 1pt off .7pt] (I234) -- (O234) (O18) -- (O90) (I162) -- (I306);
				\draw[myc2,dash pattern=on 3.6pt off 1.4pt] (I306) -- (O306) (O90) -- (O162) (I18) -- (I234);
				\draw[myc3,dash pattern=on 6.3pt off 1.8pt] (I18) -- (O18) (O162) -- (O234) (I306) -- (I90);
				\draw[myc4,dash pattern=on 11pt off 2.2pt] (I90) -- (O90) (O234) -- (O306) (I18) -- (I162);
				\draw[myc5] (I162) -- (O162) (O306) -- (O18) (I90) -- (I234);
			\end{tikzpicture}
			\caption{The edge set of $K([7] \setminus v,2)$ is partitioned into triples corresponding to Equation \eqref{A7:rel2}. Edges belonging to the same part of this partition are depicted with matching colors and dash patterns.}
			\label{figure3}
		\end{figure}
	\end{itemize}
	
	Using the preceding observations, we show that \hyperref[proof:perfecta]{Case (a)} and \hyperref[proof:perfecta]{Case (b)} each lead to a contradiction.
	
	\begin{enumerate}[label=(\alph*)]
		\item Consider the vertex-and-edge-coloring of the neighborhood $N(y)$ (see \myautorefstar{figure2}{(a)}).
		The induced subgraph on the non-neighbors of $x$ in $N(y)$ (the black vertices in \myautorefstar{figure2}{(a)}) is a $6$-cycle, and the edges of this $6$-cycle are precisely the edges of $N(y)$ that are not incident with $x$ but belong to the same partition class (see \autoref{figure3}) as one of the edges incident with $x$.
		By observation $(ii)$, we see that there have to be three red edges and three green edges in this $6$-cycle, contradicting observation $(i)$.
		
		\item Consider the vertex-and-edge-coloring of the neighborhood $N(x)$ (see \myautorefstar{figure2}{(b)}).
		Again, the induced subgraph on the non-neighbors of $y$ in $N(x)$ (the black vertices in \myautorefstar{figure2}{(b)}) is a $6$-cycle, and the edges of this $6$-cycle are precisely the edges of $N(x)$ that are not incident with $y$ but belong to the same partition class (see \autoref{figure3}) as one of the edges incident with $y$.
		By observation $(ii)$, we see that two pairs of opposite edges in the $6$-cycle have the same color and one pair of opposite edges has different colors. 
		This shows once again that there is an odd number of red edges in this $6$-cycle, contradicting observation $(i)$.
	\end{enumerate}
	It follows that there is no way to color $K(7,2)$ consistently with at least one red edge.
\end{proof}

Next, we extend \autoref{thm:A7} to $A_n$ for all $n \geq 7$.

\begin{theorem}
	\label{thm:perfectAn}
	The homogeneous solution group $\Gamma(M_{A_n},0)$ is perfect for all $n \geq 7$.
\end{theorem}
\begin{proof}
	Once again, we show that the abelianization of $\Gamma(M_{A_n},0)$ is trivial.
	This is equivalent to showing that \eqref{eqn:MH-triples} has no non-trivial solutions for $x_{h_1},\ldots,x_{h_\ell} \in \{\pm 1\}$.
	
	Let $(x_h)_{h \in O_2(A_n)} \in \{\pm 1\}^{O_2(A_n)}$ be a solution to \eqref{eqn:MH-triples}.
	Note that every element of order $2$ in $A_n$ has cycle type $(a_1a_2)(a_3a_4)\cdots (a_{4r-3}a_{4r-2})(a_{4r-1}a_{4r})$; that is, it consists of an even number of disjoint $2$-cycles.
	To prove that $x_h = 1$ for all $h \in O_2(A_n)$, we proceed by induction on the number of $2$-cycles in the cycle type of $h$.
	\begin{itemize}
		\item First consider the case that $h = (ab)(cd)$.
		Choose some subset $S \subseteq [n]$ with $|S| = 7$ such that $a,b,c,d \in S$.
		Note that the set of alternating permutations of $[n]$ that fix every point in $[n] \setminus S$ forms a subgroup $H_S \subseteq A_n$ that is isomorphic to $A_7$ and that contains $h$.
		Restricting the solution $(x_{h_1},\ldots,x_{h_\ell})$ to $H_S$ (by removing all variables that do not belong to $H_S$), we obtain a solution $(x_{h_1'},\ldots,x_{h_{\ell'}'})$ to the relations of $A_7$.
		Hence it follows from \autoref{thm:A7} that $x_{h_i'} = 1$ for all $i \in [\ell']$.
		In particular, we have $x_h = 1$.
		
		\item Let $N \geq 1$ be given such that $x_h = 1$ for all $h \in O_2(A_n)$ consisting of at most $2N$ disjoint $2$-cycles, and suppose that $h \in O_2(A_n)$ consists of $2N + 2$ disjoint $2$-cycles.
		Write $h$ in disjoint cycle notation (with the cycles in arbitrary order), and let $h' = (a_{4N+1}a_{4N+2})(a_{4N+3}a_{4N+4})$ be the last two cycles. 
		Then $\{h,h',hh'\}$ forms a triple of pairwise commuting order $2$ elements that multiply to the identity, and both $h'$ and $hh'$ consist of at most $2N$ disjoint $2$-cycles.
		It follows from the induction hypothesis that $x_{h'} = x_{hh'} = 1$.
		Moreover, by \eqref{eqn:MH-triples} we have $x_h x_{h'} x_{hh'} = 1$, so it follows that $x_h = 1$.
		\qedhere
	\end{itemize}
\end{proof}

Combining our results, we can now present a sequence of graphs that have quantum symmetry and trivial automorphism group. 

\begin{theorem}
	\label{thm:graphtrivialautgroupquantumsymmetry}
	The graph $G'(M_{A_n},b)$ from \autoref{defgraph'} has quantum symmetry and trivial automorphism group for all $n \ge 7$. 
\end{theorem}

\begin{proof}
	By \autoref{cor:non-trivial-non-commutative} and \autoref{thm:perfectAn}, the solution group $\Gamma(M_{A_n},0)$ is a non-trivial perfect group whenever $n \geq 7$.
	By \autoref{thm:allsolgroupsuncolored}, $\Qut(G'(M_{A_n},b))$ is isomorphic to the dual of the solution group $\Gamma(M_{A_n},0)$, so it follows from \autoref{lem:ab} that $G'(M_{A_n},b)$ has non-trivial quantum automorphism group and trivial (classical) automorphism group whenever $n \geq 7$.
\end{proof}

The above is of course just a more specific version of \autoref{thm:main1}, where we specify the specific graphs that have quantum symmetry but trivial automorphism group.

\begin{remark}
	\label{rmk:3.A7}
	\autoref{thm:perfectAn} proves that the solution group $\Gamma(M_{A_n},0)$ is a non-trivial perfect group, but it doesn't tell us which group it is.
	Using the computer algebra system GAP \cite{GAP}, we were able to determine that $\Gamma(M_{A_7},0)$ is isomorphic to the triple cover of $A_7$ (equivalently, the unique perfect group of order $7560$), and $\Gamma(M_{A_8},0)$ is isomorphic to $A_8$.
	In particular, $\Gamma(M_{A_7},0)$ and $\Gamma(M_{A_8},0)$ are finite, and there exists a graph which has the dual of $A_8$ as its quantum automorphism group.
	Code to reproduce these computations is provided in \cite{our-GAP-code}.
	
	\label{rmk:infinite}
	We do not know whether $G'(M_{A_n},b)$ is finite or infinite when $n \geq 9$, but we suspect that $\Gamma(M_{A_n},0) \cong A_n$ for all $n \geq 8$.
	Nevertheless, even without knowing the exact solution groups, it is possible to obtain a graph with trivial automorphism group and infinite\-/dimensional quantum automorphism group from our construction.
	For instance, consider the linear system $M_{A_7} \oplus M_{A_7}$ consisting of two copies of $M_{A_7}$ in block diagonal form.
	Then $\Gamma(M_{A_7} \oplus M_{A_7},0)$ is the free product of two copies of $\Gamma(M_{A_7},0)$, and as such is an infinite perfect group.
	Therefore $G'(M_{A_7} \oplus M_{A_7},b)$ is a finite graph with trivial automorphism group and infinite\-/dimensional quantum automorphism group.
\end{remark}

\begin{remark}
	Using the computer algebra system GAP \cite{GAP}, we also found other groups which can be used instead of $A_n$ ($n \geq 7$) in the construction.
	For example, our computations show that one could also use the sporadic Mathieu groups $M_{11}$, $M_{12}$, $M_{22}$, $M_{23}$ and $M_{24}$, and furthermore suggest that the projective special linear group $\PSL(2,q)$ also works for all odd prime powers $q \geq 19$ (tested up to $q = 349$).
	In other words, if $H$ is any one of the aforementioned groups, then $\Gamma(M_H,0)$ is a non-trivial perfect group, so that the graph $G'(M_H,b)$ has trivial automorphism group and non-trivial quantum automorphism group.
	Apart from the computations in GAP, we do not have proofs for these cases.
	Our code can be found in \cite{our-GAP-code}.
\end{remark}

\section{Two non-isomorphic, quantum isomorphic asymmetric graphs}
\label{sec:asym-qisom}

In this section, we adapt the proof of the main theorem to give an example of two non-isomorphic asymmetric graphs that are quantum isomorphic, thereby answering \cite[Problem 3.11]{Weber-survey}.
For this we use the following lemmas.

\begin{lemma}
	\label{lem:M12}
	There exists a finite group $H$ such that $\Gamma(M_H,0)$ is perfect and $Z(H)$ contains an element of order $2$.
\end{lemma}
\begin{proof}
	Both conditions can be checked in polynomial time using a computer algebra system. (Recall that the solution group $\Gamma(M,0)$ is perfect if and only if the linear system $M x = 0$ has no non-trivial classical solutions.)
	By exhaustive search, one finds that the smallest group $H$ satisfying these properties is the double cover of the Mathieu group $M_{12}$.
	(Equivalently, $H$ is the unique perfect group of order $190080$ \cite[Section 5.4]{Holt-Plesken}, and it is the automorphism group of any $12 \times 12$ Hadamard matrix \cite{Hall}.)
	Code to reproduce this computation is provided in \cite{our-GAP-code}.
\end{proof}

\begin{lemma}
	\label{lem:qiso, non-iso}
	Let $M \in \Z_2^{m\times n}$ and $b \in \Z_2^m$. If the linear system $Mx = b$ has no classical solution and one has $\gamma \neq 1$ in the solution group $\Gamma(M,b)$, then $G(M,b)$ and $G(M,0)$ are quantum isomorphic, non-isomorphic graphs.
\end{lemma}

\begin{proof}
	Assume the linear system $Mx = b$ has no classical solution and we have $\gamma \neq 1$ in the solution group $\Gamma(M,b)$. Assume there exists an isomorphism $\varphi:V(G(M,b))\to V(G(M,0))$. We will show that the restriction of $\psi:=\varphi|_{V(\hat{G}(M,b))}$ is an isomorphism between $\hat{G}(M,b)$ and $\hat{G}(M,0)$. Recall the definition of $Q_k$ from \autoref{defgraph}. Since the graphs are colored, we see $\psi((k,\alpha))\in Q_k$, and therefore we have that $\psi$ is a bijection $\psi:V(\hat{G}(M,b))\to V(\hat{G}(M,0))$. Consider $(k,\alpha), (l,\beta) \in V(\hat{G}(M,b))$. Since we know $\psi((k,\alpha))\in Q_k$ and $\psi((l,\beta))\in Q_l$, we get $i\sim_{\hat{G}(M,b)} j$ if and only if $\psi(i)\sim_{\hat{G}(M,0)}\psi(j)$. Now assume we have $(k,\alpha), (l,\beta) \in V(\hat{G}(M,b))\subset V(G(M,b))$ such that $S_k\cap S_l \neq \emptyset$, with symmetric difference $\alpha \Delta \beta$. For all $i\in S_k\cap S_l$, we know that $(k,\alpha), (l,\beta)$ are adjacent to the same $(i,\delta)$ in $G(M,b)$ if $(\alpha \Delta \beta)_i=1$ and they are adjacent to different vertices of color $i$ if $(\alpha \Delta \beta)_i=-1$. Since having a common neighbor of color $i$ is preserved by isomorphism, we see that $\psi((k,\alpha)), \psi((l,\beta))$ have symmetric difference $\alpha \Delta \beta$. Therefore, $\psi$ is an isomorphism between $\hat{G}(M,b)$ and $\hat{G}(M,b)$. This contradicts \cite[Theorem 6.2]{nonsignalling}, since $\hat{G}(M,b)$ and $\hat{G}(M,b)$ are colored versions of the graphs considered in that paper. We deduce that $G(M,b)$ and $G(M,0)$ are not isomorphic.
	
	Now, we will show that $G(M,b)$ and $G(M,0)$ are quantum isomorphic. Since we have $\gamma \neq 1$ in the solution group $\Gamma(M,b)$, there exists an operator solution to $\prod_{i\in S_k}x_i=(-1)^{b_k}$, see \cite[Theorem 4]{CLS}. Let $x_i$ be such an operator solution and define the $u=(u_{xy})_{x\in V(G(M,0), y\in V(G(M,b)}$ as follows:
	\begin{align*}
		u_{(i,s)(i,t)} &=\frac{1+stx_i}{2} && \text{ for }i\in [n],\\
		u_{(k,\alpha)(k,\beta)} &=\prod_{i\in S_k}\left(\frac{1+\alpha_i\beta_ix_i}{2}\right)&& \text{ for }k\in [m],\\
		u_{xy}&=0 &&\text{otherwise}.
	\end{align*}
	Similar to Step $2$ in the proof of \cite[Theorem 3.8]{RSsolution}, one can show that $u$ is a magic unitary. For $j\in S_k$, $st\neq \alpha_j\beta_j$, it holds
	\begin{align*}
		u_{(j,s)(j,t)}u_{(k,\alpha)(k,\beta)}&=\frac{1+stx_j}{2}\prod_{i\in S_k}\left(\frac{1+\alpha_i\beta_ix_i}{2}\right)\\
		&=\frac{1+stx_j}{2}\frac{1+\alpha_j\beta_jx_j}{2}\prod_{i\in S_k, i\neq j}\left(\frac{1+\alpha_i\beta_ix_i}{2}\right)\\
		&=0,
	\end{align*}
	since $x_i$ commute for $i\in S_k$ and $(1-x_j)(1+x_j)=0$. For all other cases with $a\sim b$ and $c\nsim d$, we get $u_{ac}u_{bd}=0$ since one of the factors is already zero. We conclude that $G(M,b)$ and $G(M,0)$ are quantum isomorphic. 
\end{proof}

\begin{theorem}
	There exists a pair of non-isomorphic, connected, asymmetric graphs that are quantum isomorphic.
\end{theorem}
\begin{proof}
	By \autoref{lem:M12}, we may choose a finite group $H$ such that $\Gamma(M_H,0)$ is perfect and $Z(H)$ contains an element of order $2$.
	Then the linear system $M_H x = 0$ has no non-trivial classical solutions, so the columns of $M_H$ are linearly independent.
	Let $J \in Z(H)$ be an element of order $2$, let $b_J$ be the column of $M_H$ corresponding to $J$, and let $M_{H,J}'$ be the matrix obtained from $H$ by removing the column $b_J$.
	
	We claim that the solution groups $\Gamma(M_H,0)$ and $\Gamma(M_{H,J}',b_J)$ have the same generators and relations, and are therefore isomorphic.
	To that end, let $\{x_h\}_{h \in O_2(H)}$ and $\{y_h\}_{h \in O_2(H) \setminus \{J\}} \cup \{\gamma\}$ be the generators of $\Gamma(M_H,0)$ and $\Gamma(M_{H,J}',b_J)$, respectively.
	We identify $x_J$ with $\gamma$ and $x_h$ with $y_h$ for all $h \neq J$.
	Then the only difference between the relations of $\Gamma(M_H,0)$ and $\Gamma(M_{H,J}',b_J)$ is that the latter solution group imposes the commutation relations $[y_h,\gamma] = 1$ for all $h \in O_2(H)$, whereas the former solution group does not necessarily impose $[x_h,x_J] = 1$ for all $h \in O_2(H)$.
	However, by \autoref{prop:MH-group}, we also have $[x_h,x_J] = 1$ for all $h \in O_2(H)$, because $[h,J] = 1$ (since $J \in Z(H)$).
	This shows that $\Gamma(M_H,0)$ and $\Gamma(M_{H,J}',b_J)$ have the same generators and relations, and are therefore isomorphic.
	
	Since the columns of $M_H$ are linearly independent, $b_J$ can not be written as a linear combination of the columns of $M_{H,J}'$.
	Therefore the linear system $M_{H,J}' x = 0$ has no non-trivial classical solutions.
	On the other hand, we have $\gamma \neq 1$ in the solution group $\Gamma(M_{H,J}',b_J)$, since the natural map $\Gamma(M_{H,J}',b_J) \to H$ maps $\gamma$ to a non-identity element.
	Hence it follows from \autoref{lem:qiso, non-iso} that the vertex-colored graphs $G(M_{H,J}',b_J)$ and $G(M_{H,J}',0)$ are quantum isomorphic but not classically isomorphic.
	By \autoref{cor:decolor}, the same is true for the decolored graphs $G'(M_{H,J}',b_J)$ and $G'(M_{H,J}',0)$.
	
	It follows from our decoloring procedure that $G'(M_{H,J}',b_J)$ and $G'(M_{H,J}',0)$ are connected.
	To see that they are also asymmetric, note that the quantum automorphism group of both of these graphs is isomorphic to the dual of $\Gamma(M_{H,J}',0)$ (by \autoref{thm:allsolgroupsuncolored}), and hence the classical automorphism group of both of these graphs is isomorphic to the dual of the abelianization of $\Gamma(M_{H,J}',0)$ (by \autoref{lem:ab}).
	But the abelianization of $\Gamma(M_{H,J}',0)$ is trivial, since $\Gamma(M_{H,J}',0) \cong \Gamma(M_{H,J}',b_J) / \langle b_J \rangle$ and $\Gamma(M_{H,J}',b_J) \cong \Gamma(M_H,0)$ is perfect (every quotient of a perfect group is perfect).
	This shows that both graphs are asymmetric.
\end{proof}

\section{Every finite group occurs as a quantum automorphism group}
\label{sec:frucht}

In \cite{Frucht}, Frucht showed that every finite group can be realized as the automorphism group of a graph. In this section, we prove that the graphs constructed in that paper have no quantum symmetry. Together with our previous results, this yields that for every finite group, we find pairs of graphs whose automorphism groups are isomorphic to that group, but one of the graphs has quantum symmetry and the other one does not. 

The following graphs were constructed in \cite{Frucht}. Note that their automorphism group coincides with the group they are constructed from. 

\begin{definition}
	\label{def:Fruchtconstruction}
	Let $\Gamma$ be a finite group of order $n\geq3$ with labelled elements $g_1, \dots g_n$. For $i,j \in [n], i\neq j$, let $m(i,j)\in \N,$ $m(i,j)\geq 2$ such that $m(i,j)=m(s,t)$ if and only if $g_ig_j^{-1}=g_sg_t^{-1}$. Define $G_{\Gamma}$ to be the following (undirected) graph:
	\begin{itemize}
		\item It has vertices $p_i$, $q_{ijk}$ and $r_{ijl}$ for $i,j\in [n]$ with $i\neq j$, for $k \in [2m(i,j)-2]$ and $l \in [2m(i,j)-1]$.
		\item They are adjacent in the following way: $p_i \sim q_{ij1}$, $p_i\sim r_{ji1}$, $q_{ij1}\sim r_{ij1}$, $q_{ijk}\sim q_{ij(k-1)}$ and $r_{ijl}\sim r_{ij(l-1)}$. 
	\end{itemize}
	For all groups of order $1\leq i \leq 2$, we define $G_{\Gamma}=K_i$.
\end{definition}

Recall that the \emph{degree} $\deg(v)$ of a vertex $v$ denotes the number of neighbors of $v$ in the graph $G$. The \emph{distance} $d(v,w)$ between two vertices $v,w \in V(G)$ is the length of a shortest path connecting $v$ and $w$. The following lemma consists of \cite[Lemma 3.2.3]{Ful} and \cite[Lemma 4.3]{RSsolution}. 

\begin{lemma}
	\label{lem:distdeg}
	Let $G$ be a finite graph and $u$ the fundamental representation of $\Qut(G)$. Let $v,w,a,b \in V(G)$. 
	\begin{itemize}
		\item[(i)] If $\deg(v)\neq \deg(w)$, then $u_{vw}=0$.
		\item[(ii)] If there exists $p \in V(G)$ with $d(v,p)=k$ such that $\deg(p) \neq \deg(q)$ for all $q \in V(G)$ with $d(w,q)=k$, then $u_{vw}=0$. 
	\end{itemize}
\end{lemma}

We need another easy lemma, which is used to shorten an upcoming proof. 

\begin{lemma}
	\label{lem:genequal}
	Let $G$ be a finite graph and $u$ the fundamental representation of $\Qut(G)$. Let $v,w,a,b \in V(G)$, $v\sim w$, $a \sim b$. If $u_{vx}=0$ for all $x\sim b, x\neq w$ and $u_{ac}=0$ for all $c\sim w$, $c\neq b$, then $u_{vw}=u_{ab}$.
\end{lemma}
\begin{proof}
	We compute
	\begin{align*}
		u_{vw}=u_{vw}\left(\sum_{c;c\sim w}u_{ac}\right)=u_{vw}u_{ab}=\left(\sum_{x;x\sim b}u_{vx}\right)u_{ab}=u_{ab},
	\end{align*}
	by using Relations \eqref{rel2} and \eqref{rel4} in \autoref{def:QAG} and the assumption. 
\end{proof}

We are now ready to show that the graphs of \autoref{def:Fruchtconstruction} have no quantum symmetry. 

\begin{theorem}
	\label{thm:Fruchtnoqsym}
	Let $\Gamma$ be a finite group and let $G_{\Gamma}$ be the graph as in \autoref{def:Fruchtconstruction}. Then $G_{\Gamma}$ has no quantum symmetry. 
\end{theorem}

\begin{proof}
	For groups of order $1$ and $2$, we know that $K_n$ has no quantum symmetry for $n=1,2$. Let $n\geq 3$. The vertices in $G_{\Gamma}$ have the following degrees:
	\begin{itemize}
		\item $\deg(p_i)=2(n-1)$, since $p_i \sim q_{ij1}$, $p_i\sim r_{ji1}$ for $j\neq i$,
		\item $\deg(q_{ij1})=3=\deg(r_{ij1})$, because $q_{ij1}\sim p_i, q_{ij2},r_{ij1}$ and $r_{ij1}\sim p_j, r_{ij2}, q_{ij1}$,
		\item $\deg(q_{ijk})=2=\deg(r_{ijl})$ for $2\leq k\leq 2m(i,j)-3$, $2\leq l\leq 2m(i,j)-2$, since $q_{ijk}\sim q_{ij(k-1)}, q_{ij(k+1)}$ and $r_{ijl}\sim r_{ij(l-1)}, r_{ij(l+1)}$,
		\item $\deg(q_{ijk})=1=\deg(r_{ijl})$ for $k=2m(i,j)-2, l=2m(i,j)-1$, as $q_{ijk}\sim q_{ij(k-1)}$ and $r_{ijl}\sim r_{ij(l-1))}$.
	\end{itemize}
	Let $u$ be the fundamental representation of $\Qut(G_{\Gamma})$. Note that $2(n-1)>3$ for $n\geq 3$. From \autoref{lem:distdeg} $(i)$, we get 
	\begin{align}
		&u_{{p_a}q_{ijb}}=u_{q_{ijb}{p_a}}=u_{{p_a}r_{ijc}}=u_{r_{ijc}{p_a}}=0, \label{eq:zero1}\\
		&u_{q_{ijx}q_{stk}}=u_{q_{stk}q_{ijx}}=u_{r_{ijy}r_{stl}}=u_{r_{stl}r_{ijy}}=0\label{eq:zero2}\\
		&u_{q_{stk}r_{ijy}}=u_{r_{ijy}q_{stk}}=u_{q_{ijx}r_{stl}}=u_{r_{stl}q_{ijx}}=0,\nonumber
	\end{align}
	for $x=2m(i,j)-2, y=2m(i,j)-1, k\neq 2m(s,t)-2, l\neq 2m(s,t)-1$. Using \autoref{lem:distdeg} $(ii)$, we furthermore see
	\begin{align}
		u_{q_{ijk}q_{stl}}=u_{r_{ijk}r_{stl}}=u_{r_{stl}r_{ij1}}=u_{q_{ijk}r_{stl}}=u_{r_{ijk}q_{stl}}=0 \text{ for } k\neq l,\label{eq:zero3}
	\end{align}
	since the distance of $q_{ijk}$ and $r_{stl}$ to a vertex of degree $3$ ($q_{ij1}$ and $r_{st1}$) is $k-1$ and $l-1$, respectively. 
	
	We obtain 
	\begin{align}
		u_{q_{ijk}q_{stk}}=u_{q_{ij(k+1)}q_{st(k+1)}}, \quad u_{r_{ijk}r_{stk}}=u_{r_{ij(k+1)}r_{st(k+1)}}, \nonumber\\
		u_{q_{ijk}r_{stk}}=u_{q_{ij(k+1)}r_{st(k+1)}}, \quad 
		u_{r_{ijk}q_{stk}}=u_{r_{ij(k+1)}q_{st(k+1)}},\label{eq:geneqal2}
	\end{align}
	by \autoref{lem:genequal}, where the assumptions for the lemma hold by Equation \eqref{eq:zero1} and Equation \eqref{eq:zero3}. This yields $u_{q_{ijk}r_{stk}}=u_{q_{ijz}r_{stz}}=0$ for $z=\mathrm{min}(2m(i,j)-1, 2m(s,t)-2)$ by Equation \eqref{eq:zero2} as $2m(i,j)-1\neq 2m(s,t)-2$. We similarly get $u_{q_{ijk}q_{stk}}=0=u_{r_{ijk}r_{stk}}$ for $m(i,j)\neq m(s,t)$. 
	
	Let $m(i,j)=m(s,t)$. Using \autoref{lem:genequal} several times, we obtain 
	\begin{align}
		u_{p_ip_s}=u_{q_{ij1}q_{st1}}=u_{r_{ij1}r_{st1}}=u_{p_jp_t}.\label{eq:geneqal1} 
	\end{align}
	From Equations \eqref{eq:zero1}--\eqref{eq:geneqal1}, we deduce that $C(\Qut(G_{\Gamma}))$ is generated by the elements $u_{p_ip_s}$. 
	
	Let $m(i,j)\neq m(s,q)$. Then there is a unique $t\neq q$ such that $m(i,j)= m(s,t)$. By Equation \eqref{eq:geneqal1}, we get
	\begin{align*}
		u_{p_ip_s}u_{p_jp_q}=u_{p_jp_t}u_{p_jp_q}=0,
	\end{align*}
	since $p_t\neq p_q$. 
	Finally, let $m(i,j)= m(s,t)$. We deduce 
	\begin{align*}
		u_{p_ip_s}u_{p_jp_t}=u_{p_jp_t}=u_{p_ip_s}u_{p_jp_t},
	\end{align*}
	from Equation \eqref{eq:geneqal1} and therefore see that all generators of $C(\Qut(G_{\Gamma}))$ commute. 
\end{proof}

\begin{remark}
	Since we know $\Aut(G_{\Gamma})\cong \Gamma$ by \cite{Frucht}, the previous theorem shows a weak quantum analog of Frucht’s theorem: every finite group can be realized as the quantum automorphism group of a finite graph. A stronger analog was recently proven in \cite{brannan2025quantumfrucht}: They show that every finite quantum group is the quantum automorphism group of a finite quantum graph.
\end{remark}

Recall the graph $G_{\Gamma}$ from \autoref{def:Fruchtconstruction}
and $G'(M_{A_n},b)$ from \autoref{defgraph} and \autoref{deflinsysgroup}. The next theorem shows that for every finite group, there is a pair of graphs such that both automorphism groups are isomorphic to that group, but exactly one of them has quantum symmetry. 

\begin{theorem}
	\label{thm:main2-specified}
	Let $\Gamma$ be a finite group and consider the graphs $G_1:=G_{\Gamma}$ and $G_2:=G_{\Gamma}\cup G'(M_{A_7},b)$. Then, we have $\Aut(G_i) \cong \Gamma$ for $i=1,2$ and $\Qut(G_1) = \Aut(G_1)$ but $G_2$ has quantum symmetry.
\end{theorem}

\begin{proof}
	It holds $\Aut(G_i) \cong \Gamma$, since we know $\Aut(G_{\Gamma}) \cong \Gamma$ from \cite{Frucht} and since $G'(M_{A_7},b)$ has trivial automorphism group by \autoref{thm:graphtrivialautgroupquantumsymmetry}. From \cite[Corollary 7.1.4]{Schmidtthesis}, we know that the quantum automorphism group of the disjoint union of non-quantum isomorphic graphs is the free product of the quantum automorphism group of the graphs. Since $G'(M_{A_7},b)$ has quantum symmetry by \autoref{thm:graphtrivialautgroupquantumsymmetry}, we see that $G_2$ also has quantum symmetry. 
\end{proof}

The previous theorem is a version of \autoref{thm:main2}, where we specify the graphs we use for every finite group. Note that the theorem shows that
there do not exist any quantum-excluding groups; that is, there is no finite group such that every graph whose automorphism group is isomorphic to that group has no quantum symmetry. This answers a question that was asked in \cite{QFrucht}.
Furthermore, we see that the quantum automorphism group of a graph can never be determined from its automorphism group, since for all automorphism groups there are both graphs with and without quantum symmetry. This answers a question raised in \cite[Section 8.1]{Schmidtthesis}.

The construction used in the proof of \autoref{thm:main2-specified} results in $G_2$ having infinite\-/dimensional quantum automorphism group whenever $\Gamma \neq 1$. We leave as an open question whether or not there exists a graph $\hat{G}_\Gamma$ such that $\Aut(\hat{G}_\Gamma) \cong \Gamma$ and $\Qut(\hat{G}_\Gamma) \not\cong \Gamma$ is finite\-/dimensional for arbitrary $\Gamma$.

\small
\interlinepenalty99 
\bibliographystyle{alphaurl}
\bibliography{graphqsymtrivialaut}

\end{document}